\documentclass[11pt]{article}

\usepackage{graphicx}  
\usepackage{amsmath}  
\usepackage{amssymb,amsthm}
\usepackage{color}

\newcommand{\Z}{{\mathbb Z}}

\newtheorem{theorem}{Theorem}[section]\newtheorem{lemma}[theorem]{Lemma}

\newtheorem{proposition}[theorem]{Proposition}
\newtheorem{corollary}[theorem]{Corollary}
\newtheorem{remark}[theorem]{Remark}

\title{\bf Bifurcations of standing localized waves \\ on periodic graphs}

\author{Dmitry Pelinovsky$^{1}$ and Guido Schneider$^{2}$ \\
{\small $^{1}$ Department of Mathematics, McMaster University, Hamilton, Ontario, L8S 4K1, Canada} \\
{\small $^{2}$ Institut f\"{u}r Analysis, Dynamik und Modellierung,
Universit\"{a}t Stuttgart,} \\
{\small Pfaffenwaldring 57, D-70569 Stuttgart, Germany }}

\begin{document}

\maketitle

\begin{abstract}
The nonlinear Schr\"{o}dinger (NLS) equation is considered on a periodic graph subject to the
Kirchhoff boundary conditions. Bifurcations of standing localized waves for frequencies lying below
the bottom of the linear spectrum of the associated stationary Schr\"{o}dinger equation are considered
by using analysis of two-dimensional discrete maps near hyperbolic fixed points. We prove existence of two
distinct families of small-amplitude standing localized waves, which are symmetric about the two
symmetry points of the periodic graph. We also prove properties of the two families, in particular,
positivity and exponential decay. The asymptotic reduction of the two-dimensional discrete map
to the stationary NLS equation on an infinite line is discussed in the context of the homogenization
of the NLS equation on the periodic graph.
\end{abstract}

\section{Introduction}

Analysis of the nonlinear Schr\"odinger (NLS) and wave equations
with double-well, multi-well, or periodic potentials constitutes a continuously developing subject \cite{Peli}.
These nonlinear partial differential equations (PDEs) have  potential applications to many realistic problems
such as Bose--Einstein condensation, nano-technology, and photonic optics.
In many applications, a specific waveguide geometry of the spatial domain suggests
the use of metric graphs as suitable way to approximate dynamics of the nonlinear PDEs
on such spatial domains \cite{Smilyansky,Noja}. At the vertex points, where different edges of the metric graphs
are connected, boundary conditions are given to define the coupling between the wave functions along the edges.
Kirchhoff boundary conditions are commonly used to ensure continuity of the wave functions and
the flow conservation through the vertex point \cite{Kuchment}.

The subject of the NLS and wave equations on metric graphs has seen many developments in the recent years.
At the rigorous mathematical level, the emphasis has been placed on the case of
star graphs, where existence, variational properties, stability, and
scattering of nonlinear waves have been studied, e.g. in \cite{ACFN0,ACFN2,Banica}.
Nonlinear waves in more complex graphs have been studied only very recently.
Variational results on the non-existence of ground states in unbounded graphs
with closed cycles is given in \cite{AST,AST2} under a set of certain topological conditions.
Bifurcation and stability of nonlinear waves on tadpole and dumbbell graphs were studied
in \cite{CFN,MP,Pelin} by using methods of bifurcation theory.

Periodic metric graphs arise in many contexts such as carbon nanotubes and graphene.
Spectral properties of the periodic graphs were studied in many details \cite{KL,KZ,Niikuni}.
Generalized Floquet--Bloch theory is introduced for periodic graphs in a similar fashion
to the study of Schr\"{o}dinger operators with bounded periodic potentials
\cite{Kuchment}. However, the periodic graphs represent a more challenging and fascinating subject
since the effective periodic potentials are defined in spaces of lower regularity.
As a result, more exotic phenomena arise such as the presence of embedded eigenvalues of infinite multiplicities
inside the Floquet--Bloch spectral bands \cite{KL,KZ}. In our recent work \cite{GilgPS},
we showed how to apply the spectral Floquet--Bloch decomposition for the periodic graphs
in order to analyze propagation of nonlinear waves on such graphs.

In nonlinear PDEs with smooth periodic potentials, localization of standing waves
is quite common for the frequencies occurring in the spectral gaps of the associated linear operators \cite{Peli}.
Mathematical justification of such standing localized waves in the smooth periodic potentials
is now well-understood in the tight-binding approximation \cite{Curtis,PS1,PS2}
and in the envelope approximation near the spectral edges \cite{BSTU06ZAMP,DPS,BoazWein}.

In the present work, we are interested to characterize standing localized waves on
periodic graphs near the bifurcation points. However, compared
to the tight-binding and envelope approximations, we would like to explore the
discrete nature of the periodic graphs. Consequently, we reduce the existence
of standing localized waves in the NLS equation on the periodic graph to the existence
of homoclinic orbits of the two-dimensional discrete map.
We will establish the equivalence between the differential equations on the periodic graphs
and the difference equations, which holds for all frequencies below the lowest spectral band
of the associated linear operator. In order to deduce definite results
on existence of standing localized waves on the periodic graphs, we will
use the proximity of the frequencies of the standing waves to those for the spectral edge.

Let us consider the following NLS equation
\begin{equation} \label{NLS}
i \partial_t u =  \partial_x^2 u + 2 |u|^2 u, \quad u(x,t) : \Gamma \times  \mathbb{R} \to \mathbb{C},
\end{equation}
on the periodic graph $\Gamma$ shown in Figure \ref{qgfig1}.
The same periodic graph and its modifications was considered
in the previous literature within the linear spectral theory of the associated
stationary Schr\"{o}dinger operator \cite{KL,KZ,Niikuni}.

\begin{figure}[htbp]
   \centering
\includegraphics[width=4in]{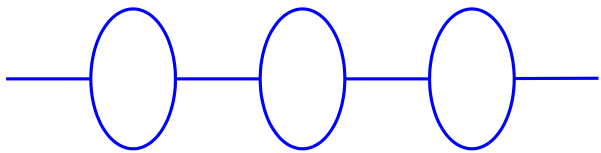}
   \caption{A schematic representation of the periodic quantum graph $\Gamma$.}
   \label{qgfig1}
\end{figure}

In the recent work \cite{GilgPS}, using the Floquet--Bloch spectral transform and energy methods, we
have addressed the time evolution problem associated with the NLS equation (\ref{NLS}) on the
periodic graph $\Gamma$ and justified the most universal approximation of the modulated wave packets
given by the homogeneous NLS equation
\begin{equation} \label{nls0}
i \partial_T \Psi = \beta \partial_X^2 \Psi + \gamma |\Psi|^2 \Psi, \quad \Psi(X,T) : \mathbb{R} \times \mathbb{R} \to \mathbb{C},
\end{equation}
where $\beta$ and $\gamma$ are specific numerical coefficients computed
by means of the Floquet--Bloch spectral theory and $\Psi$ is the envelope function
in slow spatial and temporal variables $(X,T)$
for the leading-order term in the Floquet--Bloch decomposition.

In the present work, we will consider bound states of the stationary NLS equation
\begin{equation}
\label{statNLS}
-\partial_x^2 \phi - 2 |\phi|^{2} \phi =  \Lambda \phi \qquad \Lambda \in \mathbb{R}\,,\;\phi(x) : \Gamma \to \mathbb{R},
\end{equation}
which arise for the standing waves $u(x,t) = e^{i\Lambda t} \phi(x)$ of the time-dependent NLS equation (\ref{NLS})
on the periodic graph $\Gamma$. The stationary NLS equation (\ref{statNLS}) is the Euler--Lagrange equation of the energy
functional $H_{\Lambda} := E - \Lambda Q$, where
\begin{equation}
\label{energy}
E(u) = \int_{\Gamma} |\partial_x u|^2 dx - \int_{\Gamma} |u|^4 dx
\end{equation}
and
\begin{equation}
\label{charge}
Q(u) = \int_{\Gamma} |u|^2 dx
\end{equation}
are two conserved quantities in the time evolution of the NLS equation (\ref{NLS}).
Quantities $E$ and $Q$ have the physical meaning of the Hamiltonian and mass, respectively.
In the definitions (\ref{energy}) and (\ref{charge}), the integrals are defined piecewise along each
edge of the periodic graph $\Gamma$. The critical points
of $H_{\Lambda}$ are defined in the energy space $\mathcal{E}$ given by
\begin{equation*}
\label{energy-space}
\mathcal{E} := \left\{ u \in H^1(\Gamma) : \quad u \in C^0(\overline{\Gamma}) \right\}.
\end{equation*}
Here and in what follows, $H^s(\Gamma)$, $s \in \mathbb{N}$ is defined  by using piecewise integration along each edge of the
graph $\Gamma$, whereas $u \in C^0(\overline{\Gamma})$ means that $u$ is continuous
not only along the edges but also across the vertex points of the graph $\Gamma$.

Compared to the weak energy space $\mathcal{E}$, strong solutions
of the stationary NLS equation (\ref{statNLS}) are defined
in the domain space $\mathcal{D}$, which is a subspace of $H^2(\Gamma)$
closed with the continuity conditions as well as with the Kirchhoff boundary conditions
for derivatives across the vertex points, see equations (\ref{keq1}) and (\ref{keq4}) below.
By Theorem 1.4.11 in \cite{Kuchment}, although the energy space $\mathcal{E}$ is only
defined by the continuity boundary conditions,
the Kirchhoff boundary conditions for the derivatives are natural boundary conditions
for critical points of the energy functional $H_{\Lambda}$ in the space $\mathcal{E}$.
By bootstrap arguments, any critical point of the energy functional $H_{\Lambda}$ in $\mathcal{E}$ is also
a solution of the stationary NLS equation (\ref{statNLS}) in $\mathcal{D}$.
On the other hand, solutions of the stationary NLS equation (\ref{statNLS}) in $\mathcal{D}$
are immediately the critical points of the energy functional $H_{\Lambda}$. Therefore,
the set of bound states of the stationary NLS equation (\ref{statNLS})
is equivalent to the set of critical points of the energy functional $H_{\Lambda}$.

In the context of the stationary NLS equation (\ref{statNLS}), we consider small bound states
$\phi \in \mathcal{D}$ bifurcating for small negative $\Lambda$.
In this asymptotic limit, we prove the existence of two families of small, positive, exponentially decaying bound states,
one of which is centered at the midpoint of the horizontal link connecting two rings in the periodic graph $\Gamma$
and the other one is centered symmetrically at the midpoints in the upper and lower semicircles of one ring,
see the periodic graph $\Gamma$ in Figure \ref{qgfig1}.
By discrete translational invariance, the two bound states can be translated to the midpoints of every horizontal link
and every ring in $\Gamma$.

On a technical side, we show that the two families of the bound states
can be obtained from the symmetric solutions of the two-dimensional discrete map.
The two families of bound states of the stationary NLS equation (\ref{statNLS})
on the periodic graph $\Gamma$ bifurcate from the point $\Lambda = 0$, which coincides with
the bottom of the spectrum of the linear operator $-\partial_x^2$ in $L^2(\Gamma)$.
The following theorem presents the main result of this work.

\begin{theorem}
\label{theorem-main}
There are positive constants $\Lambda_0$ and $C_0$ such that for every $\Lambda \in (-\Lambda_0,0)$,
the stationary NLS equation (\ref{statNLS}) admits two bound states $\phi \in \mathcal{D}$
(up to the discrete translational invariance) such that
\begin{equation}
\label{branch-bounds}
\| \phi \|_{H^2(\Gamma)} \leq C_0 |\Lambda|^{1/2}.
\end{equation}
One bound state satisfies
\begin{equation}
\label{branch-1}
\phi(x-L/2) = \phi(L/2-x), \quad x \in \Gamma
\end{equation}
and the other one satisfies
\begin{equation}
\label{branch-2}
\phi(x-L-\pi/2) = \phi(L+\pi/2-x), \quad x \in \Gamma,
\end{equation}
where $L$ is the length of the horizontal link and $\pi$ is the arc length of the upper and lower semicircles in $\Gamma$.
Moreover, it is true for both bound states that
\begin{itemize}
\item[(i)] $\phi$ is symmetric in upper and lower semicircles of $\Gamma$,
\item[(ii)] $\phi(x) > 0$ for every $x \in \Gamma$,
\item[(iii)] $\phi(x) \to 0$ as $|x| \to \infty$ exponentially fast.
\end{itemize}
\end{theorem}

\begin{remark}
We conjecture that the bound state satisfying (\ref{branch-1}) is the ground state of
the NLS equation (\ref{NLS}) in the sense that it is a standing wave of smallest energy $E$ at
a fixed value of $Q$ in the limit of small positive $Q$. Indeed, both bound states correspond to the values
of $\Lambda$ near the bottom of the spectrum of the linear operator $-\partial_x^2$ in $L^2(\Gamma)$.
Both bound states have properties (ii)--(iii), which are standard properties of the ground states in the NLS equation.
However, in comparison, the bound state satisfying (\ref{branch-1}) has a single hump at $x = L/2$, whereas
the bound state satisfying (\ref{branch-2}) has two humps at $x = L + \pi/2$ in the upper and lower semicircles
due to property (i). The bound state with two humps is expected to have bigger energy $E$ at a fixed value of $Q$.
Unfortunately, the energy difference between the two bound states is exponentially small
in terms of small parameter $|\Lambda|$, see Section 5 below.
\label{remark-1}
\end{remark}

\begin{remark}
Using the asymptotic method developed in \cite{MP}, one can prove existence of
two bound states of the stationary NLS equation (\ref{statNLS}) in the limit of
large negative $\Lambda$. One bound state represents a narrow solitary wave symmetric about $L/2$.
The other bound state represents two narrow solitary waves
symmetric about $L + \pi/2$ in the upper and lower semicircles. It becomes then obvious from
the methods developed in \cite{MP} that the former solution is a ground state of the NLS equation (\ref{NLS}).
Connection between the limits of small and large negative $\Lambda$ was considered numerically
in \cite{MP}. Similar results are expected for the periodic graph $\Gamma$, as
is suggested by numerical approximations in Section 5.
\label{remark-2}
\end{remark}

The rest of the paper is organized as follows. The linear spectral analysis on the periodic graph $\Gamma$
involving a two-dimensional linear discrete map is developed in Section 2. Existence of the two
bound states stated in Theorem \ref{theorem-main} is obtained in Section 3 by using
a two-dimensional nonlinear discrete map. Properties (i)--(iii) stated in Theorem \ref{theorem-main}
are proved in Section 4 by using geometric theory of stable and unstable manifolds in two-dimensional
discrete maps. Section 5 reports numerical approximations of the two bound states obtained in Theorem
\ref{theorem-main}.

\section{Linear discrete map for the spectral problem on $\Gamma$}
\label{sec-main-result}

We consider the periodic graph $\Gamma$ shown on Figure \ref{qgfig1} with the circles
of the normalized arc length $2\pi$ and the horizontal links of the length $L$. Writing the periodic graph as
$$
\Gamma = \oplus_{n \in \Z} \Gamma_n, \qquad  \textrm{with}
\qquad \Gamma_n =  \Gamma_{n,0} \oplus  \Gamma_{n,+} \oplus  \Gamma_{n,-},
$$
we can map the horizontal links to $\Gamma_{n,0} := [nP, nP + L]$
and the upper and lower semicircles to $\Gamma_{n,\pm} := [nP + L,(n+1)P]$ for $n \in \mathbb{Z}$,
where $P = L + \pi$ is the graph period. For a function $\phi : \Gamma \to \mathbb{C} $, we denote
its part on $\Gamma_{n,0}$ with $\phi_{n,0}$ and its part on $\Gamma_{n,\pm} $ with $\phi_{n,\pm}$.

The Laplacian operator $\partial_x^2$ is defined on its domain $\mathcal{D} \subset L^2(\Gamma)$
under two boundary conditions at the vertex points $\{ nP \}_{n \in \mathbb{Z}}$ and $\{ nP + L \}_{n \in \mathbb{Z}}$.
We use continuity of the functions at the vertices
\begin{eqnarray}
\left\{ \begin{array}{l} \phi_{n,0}(nP + L) = \phi_{n,+}(nP + L) = \phi_{n,-}(nP + L),  \\
\phi_{n+1,0}((n+1)P) = \phi_{n,+}((n+1) P) = \phi_{n,-}((n+1) P), \end{array} \right. \label{keq1}
\end{eqnarray}
and the Kirchhoff conditions for the derivatives of the functions at the vertices
\begin{eqnarray}
\left\{ \begin{array}{l} \partial_x \phi_{n,0}(nP + L) = \partial_x \phi_{n,+}(nP + L) + \partial_x \phi_{n,-}(nP + L), \\
\partial_x \phi_{n+1,0}((n+1)P) = \partial_x \phi_{n,+}((n+1)P) + \partial_x \phi_{n,-}((n+1)P).
\end{array} \right.
\label{keq4}
\end{eqnarray}

\begin{remark}
In the literature \cite{KL,KZ}, the periodic graph $\Gamma$ shown on Figure \ref{qgfig1} was directed
differently compared to the direction used in our work. Figure \ref{fig2} shows
two different orientations of the lower semicircle in the basic cell $\; \Gamma_0$.
The top panel corresponds to our orientation, whereas the bottom panel shows
the orientation used in \cite{KL,KZ}. The change in the direction along the lower semicircle
results in the change in the signs of the Kirchhoff boundary conditions (\ref{keq4}) but
does not change the spectral and bifurcation results.
\end{remark}

\begin{figure}[htbp]
\begin{center}
\includegraphics[scale=0.35]{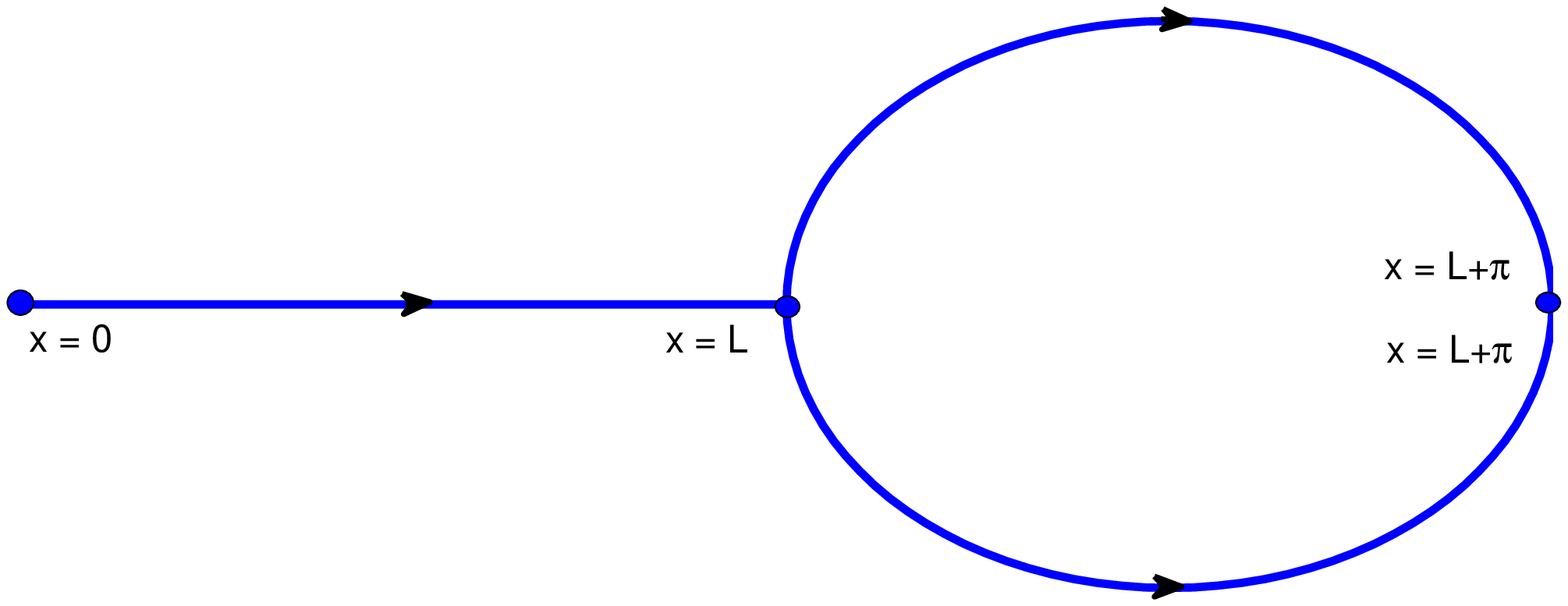}
\includegraphics[scale=0.35]{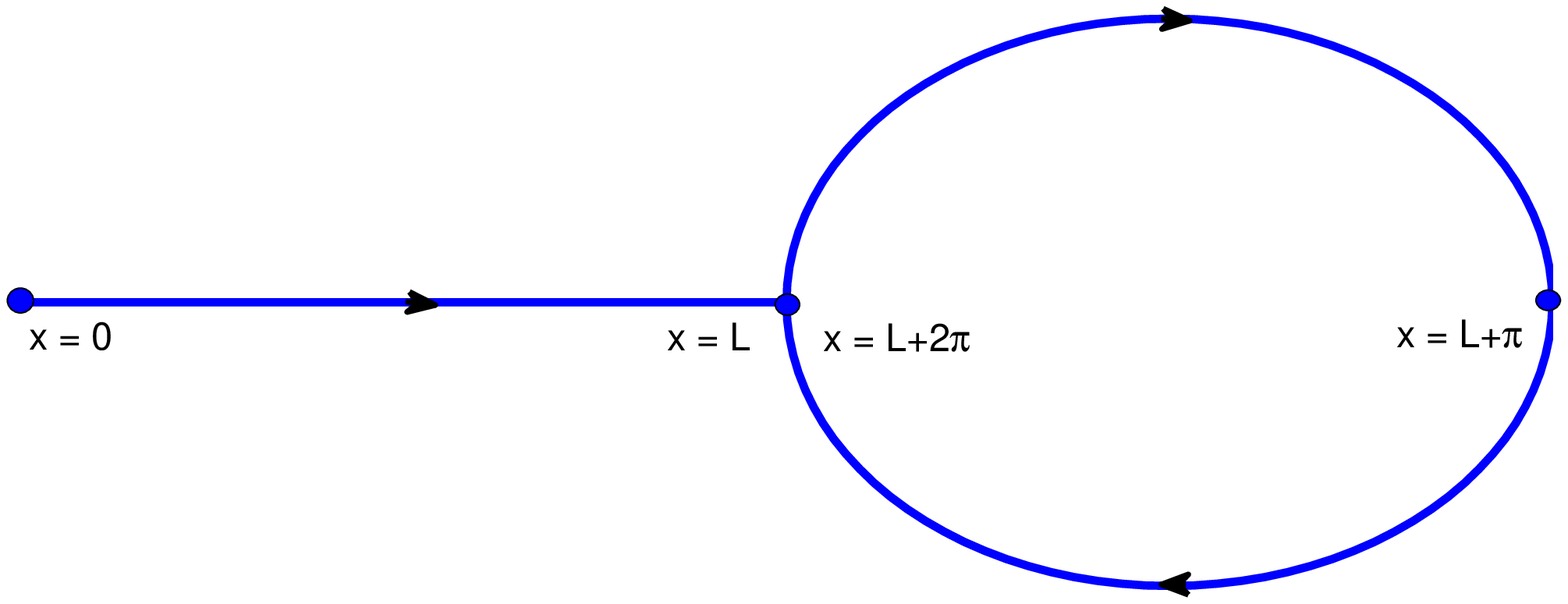}
\end{center}
\caption{The basic cell $\Gamma_0$ of the periodic graph $\Gamma$ for two different but equivalent orientations.}
\label{fig2}
\end{figure}

There exist two invariant reductions satisfying the stationary NLS equation (\ref{statNLS}).
The first reduction corresponds to the solutions compactly supported in the circles with zero
components in the horizontal links:
\begin{equation}
\label{reduction-1}
\left\{ \begin{array}{ll}
\phi_{n,0}(x) = 0, & x \in [nP, nP+L], \\
\phi_{n,+}(x) = -\phi_{n,-}(x), & x \in [nP+L,(n+1)P], \end{array} \right.
\quad \quad n \in \mathbb{Z}.
\end{equation}
The boundary conditions (\ref{keq1}) and (\ref{keq4}) are satisfied
for the reduction (\ref{reduction-1}) if and only if $\phi_{n,+}$ satisfies the homogeneous Dirichlet
boundary conditions at the end points $nP+L$ and $(n+1)P$.

The second reduction corresponds to the solution symmetrically placed in the semicircles:
\begin{equation}
\label{reduction-2}
\phi_{n,+}(x) = \phi_{n,-}(x),  \quad x \in [nP+L,(n+1)P], \quad \quad n \in \mathbb{Z}.
\end{equation}
The boundary conditions (\ref{keq1}) and (\ref{keq4}) can now be
closed in terms of only two components:
\begin{eqnarray}
\left\{ \begin{array}{l} \phi_{n,0}(nP + L) = \phi_{n,+}(nP + L),  \\
\phi_{n+1,0}((n+1)P) = \phi_{n,+}((n+1) P), \end{array} \right. \label{keq1-red}
\end{eqnarray}
and
\begin{eqnarray}
\left\{ \begin{array}{l} \partial_x \phi_{n,0}(nP + L) = 2 \partial_x \phi_{n,+}(nP + L), \\
\partial_x \phi_{n+1,0}((n+1)P) = 2 \partial_x \phi_{n,+}((n+1)P).
\end{array} \right.
\label{keq4-red}
\end{eqnarray}

The spectral problem associated with the Laplacian operator $\partial_x^2$ is given by
\begin{equation} \label{Laplacian}
- \partial_x^2 w = \lambda w, \quad x \in \Gamma.
\end{equation}
By Theorem 1.4.4 in  \cite{Kuchment}, the Laplacian operator $\partial_x^2 : \mathcal{D} \to L^2(\Gamma)$
is self-adjoint. Therefore, the values of $\lambda$ are real. Moreover, integrating by parts and
using the boundary conditions (\ref{keq1})--(\ref{keq4}), we confirm that
for every $w \in \mathcal{D} \subset L^2(\Gamma)$, we have
$$
\lambda \| w \|_{L^2(\Gamma)}^2 = \| \partial_x w \|_{L^2(\Gamma)}^2 \geq 0,
$$
hence, the values of $\lambda$ are positive. Now, we note an elementary result.

\begin{proposition}
\label{prop-spectrum}
The spectrum $\sigma(-\partial_x^2)$ in $L^2(\Gamma)$ consists of two parts, which correspond
to eigenfunctions $w \in \mathcal{D} \subset L^2(\Gamma)$,
which either satisfy the reduction (\ref{reduction-1}) or the reduction (\ref{reduction-2}).
\end{proposition}

\begin{proof}
As the linear superposition principle can be applied to the linear homogeneous equation (\ref{Laplacian}),
a general solution of the second-order differential equation (\ref{Laplacian})
on the periodic graph $\Gamma$ can be superposed as the sum of two components,
one satisfies the reduction (\ref{reduction-1}) and the other one satisfies the reduction (\ref{reduction-2}).

Indeed, for general $w_{n,+}$ and $w_{n,-}$, we can present these functions as a sum
of symmetric and anti-symmetric components (the latter ones vanish at the end points
of the intervals $[nP+L,(n+1)P]$). Due to the boundary conditions (\ref{keq1}) and (\ref{keq4}),
the two components generate the corresponding decomposition of $w_{n,0}$. The symmetric
part of $w_{n,0}$ satisfies (\ref{keq1-red}) and (\ref{keq4-red}), whereas
the anti-symmetric part of $w_{n,0}$ is identically zero, due to uniqueness of the zero
solution of the second-order differential equation (\ref{Laplacian}) with zero values
both for functions and their derivatives.
\end{proof}

By Proposition \ref{prop-spectrum}, we can search for the eigenfunctions $w$ of the
spectral problem (\ref{Laplacian}) separately within the reductions (\ref{reduction-1}) and (\ref{reduction-2}).
Eigenfunctions satisfying the reduction (\ref{reduction-1}) are given by
\begin{equation}
\label{compact}
\left\{ \begin{array}{ll} w_{n,0}(x) = 0, & x \in [nP, nP+L], \\
w_{n,\pm}(x) = \pm \delta_{n,k} \sin(m (x - 2 \pi n)), & x \in [nP+L,(n+1)P], \end{array} \right. \quad
n \in \mathbb{Z},
\end{equation}
for fixed $m \in \mathbb{N}$ and $k \in \mathbb{Z}$. There exist countably many eigenfunctions
(\ref{compact}) for the same eigenvalue $\lambda = m^2$.
Hence, the first part of $\sigma(-\partial_x^2)$ is given by the sequence of eigenvalues
$\{ m^2 \}_{m \in \mathbb{N}}$ of infinite multiplicity.

Eigenfunctions of the differential equation (\ref{Laplacian})
satisfying the reduction (\ref{reduction-2}) can be represented
in the piecewise form
\begin{equation*}
\left\{ \begin{array}{ll}
w_{n,0}(x) = a_n \cos(\omega (x - nP)) + b_n \sin(\omega (x-nP)), & x \in [nP, nP+L], \\
w_{n,\pm}(x) = c_n \cos(\omega (x-nP-L)) + d_n \sin(\omega (x-nP-L)), & x \in [nP+L,(n+1)P],
\end{array} \right.
\end{equation*}
where the spectral parameter $\lambda$ is parameterized as $\lambda = \omega^2$,
and the coefficients $\{ a_n,b_n,c_n,d_n \}_{n \in \mathbb{Z}}$ are to be defined from
the homogeneous Kirchhoff conditions (\ref{keq1-red}) and (\ref{keq4-red}).
Boundary conditions at the vertices $\{ nP+L \}_{n \in \mathbb{Z}}$ yield
\begin{equation}
\left\{ \begin{array}{l} c_n = a_n \cos(\omega L) + b_n \sin(\omega L), \\
2 d_n = - a_n \sin(\omega L) + b_n \cos(\omega L), \end{array} \right.
\label{system-1}
\end{equation}
whereas the boundary conditions at the vertices $\{ nP \}_{n \in \mathbb{Z}}$
yield
\begin{equation}
\left\{ \begin{array}{l} a_{n+1} = c_n \cos(\omega \pi) + d_n \sin(\omega \pi), \\
b_{n+1} = - 2 c_n \sin(\omega \pi) + 2 d_n \cos(\omega \pi). \end{array} \right.
\label{system-2}
\end{equation}
Eliminating $\{ c_n, d_n \}_{n \in \mathbb{Z}}$ from system (\ref{system-1}) and (\ref{system-2}),
we obtain the two-dimensional linear discrete map in the matrix form
\begin{equation}
\label{system-3}
\left[ \begin{array}{l} a_{n+1} \\ b_{n+1} \end{array} \right] = M(\omega)
\left[ \begin{array}{l} a_{n} \\ b_{n} \end{array} \right],
\end{equation}
where the monodromy matrix is given explicitly by
\begin{equation}
M(\omega) := \left[ \begin{array}{cc} \cos(\omega \pi) & \sin(\omega \pi) \\ -2 \sin(\omega \pi) & 2 \cos(\omega \pi) \end{array} \right]
\left[ \begin{array}{cc} \cos(\omega L) & \sin(\omega L) \\ -\frac{1}{2} \sin(\omega L) & \frac{1}{2} \cos(\omega L) \end{array} \right].
\end{equation}
By direct computation, we check that $\det(M) = 1$ and ${\rm tr}(M) \equiv T$ is given by
\begin{equation}
\label{transcendental}
T(\omega) = 2 \cos(\omega \pi) \cos(\omega L) - \frac{5}{2} \sin(\omega \pi) \sin(\omega L).
\end{equation}

\begin{figure}[htbp]
   \centering
   \includegraphics[width=4in]{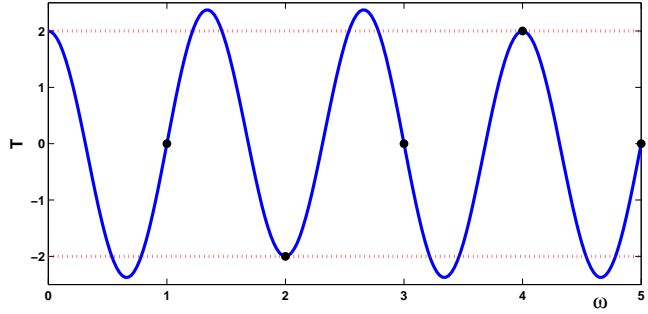}
\caption{The graph of $T$ versus $\omega$ for the periodic graph $\Gamma$ with $L = \pi/2$.}
\label{fig:example}
\end{figure}

Let $\mu_1$ and $\mu_2$ be the two eigenvalues of $M$ called the Floquet multipliers. Then,
$\mu_1 \mu_2 = 1$ and $\mu_1 + \mu_2 = T$, so that we can find the spectral bands
of the spectral problem (\ref{Laplacian}) in terms of the parameter $\lambda = \omega^2$
from the condition $|\mu_1| = |\mu_2| = 1$. The standard way of finding these bands is
to plot the graph of $T$ versus $\omega$ and to find intervals in $\omega$,
where $|T(\omega)| \leq 2$. Figure \ref{fig:example} shows this graph for $L = \pi/2$.
The eigenvalues of infinite multiplicities are shown in black dots. As we can see,
some spectral bands are disjoint from each other with nonempty gaps, whereas some others
touch each other and admit no gap.

Summarizing the previous computations, we identify the spectrum $\sigma(-\partial_x^2)$ on $L^2(\Gamma)$.

\begin{proposition}
\label{prop-location}
The spectrum $\sigma(-\partial_x^2)$ in $L^2(\Gamma)$ consists of eigenvalues $\{ m^2 \}_{m \in \mathbb{N}}$
of infinite multiplicity and a countable set of spectral bands $\{ \sigma_k \}_{k \in \mathbb{N}}$,
which are determined from the condition $T(\omega) \in [-2,2]$, where $T(\omega)$ is given by (\ref{transcendental})
and $\lambda = \omega^2$. Moreover, the eigenvalues of infinite multiplicity belong to the spectral bands:
$$
m^2 \in \cup_{k \in \mathbb{N}} \sigma_k \quad \mbox{\rm for every} \;\; m \in \mathbb{N}.
$$
\end{proposition}

\begin{proof}
The first assertion follows from Proposition \ref{prop-spectrum} and the explicit construction of eigenfunctions
described above.
In particular, the $k$-th spectral band $\sigma_k$ can be parameterized by a continuous parameter $\theta \in [-\pi,\pi]$,
which arises in the band-limited Fourier transform for bounded solutions of the linear difference map (\ref{system-3}):
\begin{equation}
\label{Fourier}
\left[ \begin{array}{l} a_{n} \\ b_{n} \end{array} \right] = \int_{-\pi}^{\pi}
\left[ \begin{array}{l} \hat{a}(\theta) \\ \hat{b}(\theta) \end{array} \right] e^{i n \theta} d \theta, \quad n \in \mathbb{Z},
\end{equation}
such that $T(\omega) = 2 \cos(\theta) \in [-2,2]$.

It remains to prove that eigenvalues $\{ m^2 \}_{m \in \mathbb{N}}$ belong to the
union of spectral bands $\cup_{k \in \mathbb{N}} \sigma_k$. Indeed, substituting
$\omega = m$ for $m \in \mathbb{N}$ to (\ref{transcendental}), we obtain
$$
T(m) = 2 (-1)^m \cos(m L) \in [-2,2].
$$
If $|\cos(mL)| = 1$, then $m^2$ is located at the spectral edge (one of the two end points of the
corresponding spectral band), see $m = 2, 4$ in Figure \ref{fig:example}.
If $|\cos(mL)| < 1$, then $m^2$ belongs to the interior of the corresponding spectral band,
see $m = 1,3,5$ in Figure \ref{fig:example}.
\end{proof}

\begin{remark}
\label{remark-lowest-inf}
It follows from the graph of $T(\omega)$ defined by  (\ref{transcendental}) that the smallest eigenvalue of infinite multiplicity $\lambda = 1$
belongs to the second spectral band if $L \in (0,\pi)$, the third spectral band if $L \in (\pi,2\pi)$, and so on.
\end{remark}

Let us simplify $T(\omega)$ defined by  (\ref{transcendental})
at the lowest end point $\lambda = 0$ of the lowest spectral band $\sigma_1$ of the spectral problem (\ref{Laplacian}). Expanding
$T$ in powers of $\omega$, we obtain
$$
T(\omega) = 2 - \nu^2 \omega^2 + \mathcal{O}(\omega^4) \quad \mbox{\rm as} \quad \omega \to 0,
$$
where $\nu^2 := \pi^2 + L^2 + \frac{5}{2} \pi L = (L+\pi/2) (L + 2 \pi)$ is a numerical constant.

Using the band-limited Fourier transform (\ref{Fourier}), one can parameterize
the Floquet multipliers on the unit circle by $\mu_1 = e^{i \theta}$ and $\mu_2 = e^{-i \theta}$.
Therefore, for small $\theta$, we obtain another expansion
$$
\mu_1 + \mu_2 = 2 \cos(\theta) = 2 - \theta^2 + \mathcal{O}(\theta^4) \quad \mbox{\rm as} \quad \theta \to 0.
$$

Bringing $T(\omega) = \mu_1 + \mu_2 = 2 \cos(\theta)$ together, we obtain the asymptotic
approximation for the lowest spectral band of the spectral problem (\ref{Laplacian})
near the lowest end point $\lambda = 0$ by
\begin{equation}
\label{lowest-band}
\lambda(\theta) = \nu^{-2} \theta^2 + \mathcal{O}(\theta^4) \quad \mbox{\rm as} \quad \theta \to 0,
\end{equation}
where $\lambda(\theta)$ is a parametrization of $\sigma_1$.

\begin{remark}
The spectral bands of the spectral problem (\ref{Laplacian})
can also be parameterized by the Bloch quasi-momentum in the Bloch wave representation of the eigenfunctions $w$.
The Bloch wave representation is known for the periodic metric graphs
\cite{Kuchment} and has been explored in our recent work \cite{GilgPS}. In the present paper,
we avoid the Bloch wave representation and work entirely with the two-dimensional discrete maps,
which generalize the linear discrete map (\ref{system-3}).
\end{remark}

\section{Nonlinear discrete map for the bound state on $\Gamma$}

The bound states defined by the stationary NLS equation (\ref{statNLS}) in $\mathcal{D} \subset H^2(\Gamma)$
may bifurcate from the zero states, when the parameter $\Lambda$ tends to the extremal points
in $\sigma(-\partial_x^2)$. Among all possible bound states, we are interested
in the small bound states bifurcating from the bottom in $\sigma(-\partial_x^2)$.

It follows from (\ref{lowest-band}) that the lowest spectral band extends from $\lambda = 0$ to positive
values of $\lambda$.
Therefore, we shall now consider bound states of the stationary NLS equation (\ref{statNLS})
for small negative $\Lambda$. We hence set $\Lambda := - \epsilon^2$ and consider
solutions of the stationary NLS equation
\begin{equation}
\label{statNLS-eps}
\partial_x^2 \phi - \epsilon^2 \phi + 2 |\phi|^{2} \phi =  0, \qquad \phi \in \mathcal{D} \subset L^2(\Gamma).
\end{equation}

\begin{remark}
As follows from the construction of Propositions \ref{prop-spectrum} and \ref{prop-location},
the lowest spectral band corresponds to
eigenfunctions of $\partial_x^2$ in $\mathcal{D}$ satisfying
the reduction (\ref{reduction-2}). By Remark \ref{remark-lowest-inf},
the lowest spectral band is disjoint from the lowest eigenvalue
$\lambda = 1$ of infinite multiplicity that corresponds to the
eigenfunctions of $\partial_x^2$ in $\mathcal{D}$ satisfying the reduction (\ref{reduction-1}).
Therefore, it is sufficient to consider
solutions of the stationary equation (\ref{statNLS-eps}) satisfying the reduction (\ref{reduction-2}).
\end{remark}

Our first task is to reduce the second-order differential equation (\ref{statNLS-eps})
on the periodic graph $\Gamma$ to the two-dimensional discrete map. To do so, let us define a solution
$\psi(x;a,b,\epsilon)$ of the initial-value problem on the infinite line:
\begin{equation}
\label{initial-value}
\left\{ \begin{array}{l} \partial_x^2 \psi - \epsilon^2 \psi + 2 |\psi|^{2} \psi =  0, \qquad x \in \mathbb{R}, \\
\psi(0) = a, \\ \partial_x \psi(0) = b, \end{array} \right.
\end{equation}
where $(a,b)$ are some real-valued coefficients. The following result is well-known for
the spatial dynamical system related to the focusing NLS equation.

\begin{proposition}
\label{prop-ODE}
For every $(a,b) \in \mathbb{R}^2$ and every $\epsilon \in \mathbb{R}$,
there is a unique global bounded solution $\psi(x) \in C^{\infty}(\mathbb{R})$ of the initial-value
problem (\ref{initial-value}).
\end{proposition}

\begin{proof}
By the standard Picard methods, a unique local real-valued $C^1$ solution $\psi$ exists
for every $(a,b) \in \mathbb{R}^2$ and every $\epsilon \in \mathbb{R}$.
By the conservation of the first-order invariant,
\begin{equation}
\label{first-order-invariant}
E := \left( \partial_x \psi \right)^2 - \epsilon^2 \psi^2 + \psi^4 = {\rm const},
\end{equation}
the local $C^1$ solution is extended as a global solution $\psi(x) \in C^1(\mathbb{R})$.
Moreover, for every $x \in \mathbb{R}$, both $\psi(x)$ and $\psi'(x)$ remain bounded.
By bootstrap arguments, the $C^1$ solution $\psi$ is now extended as a smooth solution in $x$
since the vector field of the differential equation (\ref{initial-value}) is smooth in real variable $\psi$.
\end{proof}

Using the translational invariance and the unique global solution $\psi(x) \in C^{\infty}(\mathbb{R})$
of the initial-value problem (\ref{initial-value}) given by Proposition \ref{prop-ODE},
we can now solve the following initial-value problems on finite intervals
for every $n \in \mathbb{Z}$:
\begin{equation}
\label{initial-value-1}
\left\{ \begin{array}{l} \partial_x^2 \phi_{n,0} - \epsilon^2 \phi_{n,0} + 2 |\phi_{n,0}|^{2} \phi_{n,0} =  0, \qquad x \in [nP,nP+L], \\
\phi_{n,0}(nP) = a_n, \\ \partial_x \phi_{n,0}(nP) = b_n \end{array} \right.
\end{equation}
and
\begin{equation}
\label{initial-value-2}
\left\{ \begin{array}{l} \partial_x^2 \phi_{n,+} - \epsilon^2 \phi_{n,+} + 2 |\phi_{n,+}|^{2} \phi_{n,+} =  0, \qquad x \in [nP+L,(n+1)P], \\
\phi_{n,+}(nP+L) = c_n, \\ \partial_x \phi_{n,+}(nP+L) = d_n \end{array} \right.
\end{equation}
where $\{ a_n, b_n, c_n, d_n \}$ are some real-valued coefficients. By the existence and uniqueness result of Proposition
\ref{prop-ODE} and the translational invariance, we obtain unique solutions of the initial-value problems
(\ref{initial-value-1}) and (\ref{initial-value-2}) in the following form:
\begin{equation}
\label{correspondence-phi}
\phi_{n,0}(x) = \psi(x-nP;a_n,b_n,\epsilon), \quad \phi_{n,+}(x) = \psi(x-nP-L;c_n,d_n,\epsilon).
\end{equation}

\begin{remark}
The function $\psi$ for the initial-value problem (\ref{initial-value}) can be expressed explicitly in
terms of Jacobi elliptic functions. This approach was used in the literature \cite{Waltner}
to obtain some information about standing waves on various metric graphs. In our approach,
we avoid Jacobi elliptic functions and rely on the general perturbation theory in small $\epsilon$.
The same general approach can also be applied to other nonlinear problems (e.g. with higher-order
power functions), where explicit solutions of the initial-value problem (\ref{initial-value})
are not available.
\end{remark}

Keeping in mind the reduction (\ref{reduction-2}), we can now satisfy the boundary
conditions (\ref{keq1-red}) and (\ref{keq4-red}) by using the unique solutions given by (\ref{correspondence-phi}).
Boundary conditions at the vertices $\{ nP+L \}_{n \in \mathbb{Z}}$ yield
\begin{equation}
\left\{ \begin{array}{l} c_n = \psi(L;a_n,b_n,\epsilon), \\
2 d_n = \partial_x \psi(L;a_n,b_n,\epsilon), \end{array} \right.
\label{nonlinear-system-1}
\end{equation}
whereas the boundary conditions at the vertices $\{ nP \}_{n \in \mathbb{Z}}$
yield
\begin{equation}
\left\{ \begin{array}{l} a_{n+1} = \psi(\pi;c_n,d_n,\epsilon), \\
b_{n+1} = 2 \partial_x \psi(\pi;c_n,d_n,\epsilon). \end{array} \right.
\label{nonlinear-system-2}
\end{equation}
Eliminating $\{ c_n, d_n \}_{n \in \mathbb{Z}}$ from system (\ref{nonlinear-system-1}) and (\ref{nonlinear-system-2}),
we obtain the two-dimensional nonlinear discrete map in the form
\begin{equation}
\label{nonlinear-system-3}
\left[ \begin{array}{l} a_{n+1} \\ b_{n+1} \end{array} \right] =
\left[ \begin{array}{l} \psi(\pi;\psi(L;a_n,b_n,\epsilon),\frac{1}{2} \partial_x \psi(L;a_n,b_n,\epsilon),\epsilon) \\
2 \partial_x \psi(\pi;\psi(L;a_n,b_n,\epsilon),\frac{1}{2} \partial_x \psi(L;a_n,b_n,\epsilon),\epsilon) \end{array} \right].
\end{equation}

Next we establish the reversibility of the two-dimensional discrete map (\ref{nonlinear-system-3})
about two natural centers of symmetries. The first symmetry
\begin{equation}
\label{symmetry-1}
a_{-n} = c_{n}, \quad b_{-n} = -2 d_n, \quad n \in \mathbb{Z},
\end{equation}
corresponds to the solution $\phi \in \mathcal{D}$ of the stationary equation (\ref{statNLS-eps})
symmetric about the midpoint $x = L/2$ in the $0$-th central link. The other symmetry
\begin{equation}
\label{symmetry-2}
a_{-n} = c_{n+1}), \quad b_{-n} = -2 d_{n+1}, \quad n \in \mathbb{Z},
\end{equation}
corresponds to the solution $\phi \in \mathcal{D}$ of the stationary equation (\ref{statNLS-eps})
symmetric about the midpoint $x = L + \pi/2$ in the $0$-th circle. Although the
constraints (\ref{symmetry-1}) and (\ref{symmetry-2}) involve infinitely many relations
between solutions of the nonlinear discrete map (\ref{nonlinear-system-3}),
we show that each symmetry is satisfied with only one constraint.

\begin{proposition}
\label{prop-symmetry}
The symmetry (\ref{symmetry-1}) on the solutions of the nonlinear discrete map (\ref{nonlinear-system-3})
is satisfied if and only if $(a_0,b_0)$ satisfies the following reversibility constraint:
\begin{equation}
\label{symmetry-constraint-1}
\partial_x \phi_{0,0}(L/2) = \partial_x \psi(L/2;a_0,b_0,\epsilon) = 0.
\end{equation}
The symmetry (\ref{symmetry-2}) is satisfied if and only if $(c_0,d_0)$ satisfies the following
reversibility condition:
\begin{equation}
\label{symmetry-constraint-2}
\partial_x \phi_{0,+}(L+\pi/2) = \partial_x \psi(\pi/2;c_0,d_0,\epsilon) = 0.
\end{equation}
\end{proposition}

\begin{proof}
Since the periodic graph $\Gamma$ is symmetric about the point $x = L/2$
and the stationary NLS equation (\ref{statNLS-eps}) involves only second-order derivatives,
the existence and uniqueness theory for differential equations implies that
the solution $\phi \in \mathcal{D}$ of the stationary NLS equation (\ref{statNLS-eps})
is symmetric about the point $x = L/2$ if and only if
it satisfies the condition (\ref{symmetry-constraint-1}). By the constructions
of the solution $\phi \in \mathcal{D}$ with the explicit formula (\ref{correspondence-phi}),
the symmetry (\ref{symmetry-constraint-1}) translates uniquely to the symmetry (\ref{symmetry-1})
on solutions of the discrete map (\ref{nonlinear-system-3}). Thus,
solutions of the two-dimensional discrete map (\ref{nonlinear-system-3})
satisfy (\ref{symmetry-1}) if and only if $(a_0,b_0)$ satisfy (\ref{symmetry-constraint-1}).

The statement for the symmetry (\ref{symmetry-2}) and the constraint (\ref{symmetry-constraint-2})
on $(c_0,d_0)$ is proved from the symmetry of the periodic graph $\Gamma$ about the point $x = L + \pi/2$.
\end{proof}

Although the discrete map (\ref{nonlinear-system-3}) can be used for every solution
of the stationary equation (\ref{statNLS-eps}), the results are not so explicit. Moreover,
many solution branches may coexist for the same values of $\epsilon \in \mathbb{R}$.
Therefore, we simplify the consideration for small solutions $\phi \in \mathcal{D}$
corresponding to small values of $\epsilon \in \mathbb{R}$. This simplification is based on
the following approximation result.

\begin{lemma}
\label{proposition-expansion}
Consider the initial-value problem (\ref{initial-value}) with the scaled initial conditions
\begin{equation}
\label{scaling-a-b}
a = \epsilon \alpha, \quad b = \epsilon^2 \beta
\end{equation}
where $\alpha$ and $\beta$ are some $\epsilon$-independent real-valued coefficients. For every $x_0 > 0$,
there exists $\epsilon_0 > 0$ such that the initial-value problem (\ref{initial-value})
for every $\epsilon \in (0,\epsilon_0)$ admits a unique solution $\psi(x) \in C^{\infty}(0,x_0)$,
which is smooth in $\epsilon$ and satisfies the power series expansion
\begin{equation}
\label{expansion-in-eps}
\psi(x;\epsilon \alpha, \epsilon^2 \beta, \epsilon) = \epsilon \left[ \alpha +
\epsilon \beta x + \frac{1}{2} \epsilon^2 \alpha (1 - 2 \alpha^2) x^2 + \frac{1}{6} \epsilon^3 (1 - 6 \alpha^2) \beta x^3
+ \mathcal{O}_{L^{\infty}(0,x_0)}(\epsilon^4) \right].
\end{equation}
\end{lemma}

\begin{proof}
Using the scaling transformation (\ref{scaling-a-b}), we scale the unique real-valued
solution of the initial-value problem (\ref{initial-value}) by
$\psi(x) = \epsilon \varphi(x)$ and obtain
\begin{equation}
\label{initial-value-5}
\left\{ \begin{array}{l} \partial_x^2 \varphi = \epsilon^2 (1 - 2 \varphi^2) \varphi, \\
\varphi(0) = \alpha, \\ \partial_x \varphi(0) = \epsilon \beta. \end{array} \right.
\end{equation}
From smoothness of the initial-value problem (\ref{initial-value-5}) in $\epsilon$,
we have smoothness of the unique global solution $\varphi(x) \in C^{\infty}(\mathbb{R})$ in $\epsilon$.
Therefore, the unique solution of the initial-value problem (\ref{initial-value-5}) satisfies
the regular power series expansion given by
\begin{equation}
\label{expansions-varphi}
\varphi(x;\epsilon) = \varphi_0(x) + \epsilon^2 \varphi_2(x) + \epsilon^4 \tilde{\varphi}_{\epsilon}(x),
\end{equation}
where $\varphi_0''(x) = 0$, $\varphi_2''(x) = (1-2\varphi_0^2) \varphi_0$ and
\begin{eqnarray*}
\tilde{\varphi}_{\epsilon}''(x) & = & (1-6 \varphi_0^2) \varphi_2 +
\epsilon^2 (1 - 6 \varphi_0^2 - 12 \epsilon^2 \varphi_0 \varphi_2 - 6 \epsilon^4 \varphi_2^2) \tilde{\varphi}_{\epsilon} \\
& \phantom{t} & - 6 \epsilon^2 \varphi_0 \varphi_2^2 - 2 \epsilon^4 \varphi_2^3 - 6 \epsilon^6 \varphi_0 \tilde{\varphi}_{\epsilon}^2
- 6 \epsilon^8 \varphi_2 \tilde{\varphi}_{\epsilon}^2 - 2 \epsilon^{10} \tilde{\varphi}_{\epsilon}^3.
\end{eqnarray*}
From the initial values, we have the unique expressions
for $\varphi_0(x) = \alpha + \epsilon \beta x$ and
$$
\varphi_2(x) = \frac{1}{2} (1 - 2\alpha^2) \alpha x^2 + \frac{1}{6} \epsilon (1 - 6 \alpha^2) \beta x^3 - \frac{1}{2} \epsilon^2 \alpha \beta^2 x^4
- \frac{1}{10} \epsilon^3 \beta^3 x^5.
$$
Also, by standard Gronwall's inequality, we obtain that
$\tilde{\varphi}_{\epsilon}(x)$ is bounded in $L^{\infty}(0,x_0)$ for every $x_0 > 0$ as $\epsilon \to 0$.
Substituting expressions for $\varphi_0$ and $\varphi_2$ in the power series expansion (\ref{expansions-varphi})
and neglecting the $\mathcal{O}_{L^{\infty}(0,x_0)}(\epsilon^4)$ terms, we obtain (\ref{expansion-in-eps}).
\end{proof}

By using scaling (\ref{scaling-a-b}) and expansion (\ref{expansion-in-eps}),
we introduce the scaling transformation for solutions of the discrete maps
(\ref{nonlinear-system-1}) and (\ref{nonlinear-system-2}):
\begin{equation}
\label{scaling-map}
a_n = \epsilon \alpha_n, \quad b_n = \epsilon^2 \beta_n, \quad
c_n = \epsilon \gamma_n, \quad d_n = \epsilon^2 \delta_n, \quad
n \in \mathbb{Z}.
\end{equation}
Using the connection formulas
\begin{equation}
\label{gamma-delta-in-terms-of-alpha-beta}
\left\{ \begin{array}{l} \gamma_n = \alpha_n +
\epsilon \beta_n L + \frac{1}{2} \epsilon^2 \alpha_n (1 - 2 \alpha_n^2) L^2 + \mathcal{O}(\epsilon^3), \\
2 \delta_n = \beta_n + \epsilon \alpha_n (1 - 2 \alpha_n^2) L + \frac{1}{2} \epsilon^2 \beta_n (1 - 6 \alpha_n^2) L^2 + \mathcal{O}(\epsilon^3),
\end{array} \right.
\end{equation}
we rewrite the discrete map (\ref{nonlinear-system-3}) in the explicit asymptotic form:
\begin{equation}
\label{nonlinear-system-4}
\left\{ \begin{array}{l} \alpha_{n+1} = \alpha_n + \epsilon (L + \pi/2) \beta_n +
\frac{1}{2} \epsilon^2 (L^2 + \pi L + \pi^2) (1 - 2 \alpha_n^2) \alpha_n \\
\phantom{texttexttext}
+ \frac{1}{12} \epsilon^3 (2 L^3 + 3 L^2 \pi + 6 L \pi^2 + \pi^3) (1 - 6 \alpha_n^2) \beta_n
+ \mathcal{O}(\epsilon^4), \\
\beta_{n+1} = \beta_n + \epsilon (L + 2 \pi) (1 - 2 \alpha_n^2) \alpha_n
+ \frac{1}{4} \epsilon^2 (2 L^2 + 4 L \pi + \pi^2) (1 - 6 \alpha_n^2) \beta_n + \mathcal{O}(\epsilon^3). \end{array} \right.
\end{equation}

Next, we construct a suitable approximation for solutions of the discrete map (\ref{nonlinear-system-4}).
In particular, we consider a slowly varying solution in the form
\begin{equation}
\label{slowly-varying}
\alpha_n = A(X), \quad \beta_n = B(X), \quad X = \epsilon n, \quad n \in \mathbb{Z},
\end{equation}
with $A(X), B(X) \in C^{\infty}(\mathbb{R})$. By substituting (\ref{slowly-varying})
into (\ref{nonlinear-system-4}), using Taylor series expansions, and truncating at the leading-order terms, we obtain
\begin{equation}
\label{nonlinear-system-5}
\left\{ \begin{array}{l} A'(x) = (L + \pi/2) B(x), \\
B'(x) = (L + 2 \pi) (1 - 2 A^2) A, \end{array} \right.
\end{equation}
which is equivalent to the second-order differential equation
\begin{equation}
\label{nonlinear-system-6}
A''(x) = \nu^2 (1 - 2 A^2) A, \quad \nu^2 := (L + \pi/2)(L + 2\pi).
\end{equation}
System (\ref{nonlinear-system-5}) is satisfied with the exact localized solution
\begin{equation}
\label{approximation-soliton}
A(X) = {\rm sech}(\nu X), \quad B(X) = - \mu \tanh(\nu X) {\rm sech}(\nu X), \quad X \in \mathbb{R},
\end{equation}
where $\mu^2 := (L + 2\pi)(L + \pi/2)$. In order to prove persistence of the approximation (\ref{slowly-varying}) and
(\ref{approximation-soliton}) among the reversible solutions of the discrete map (\ref{nonlinear-system-4}),
we need the following result.

\begin{proposition}
\label{proposition-inversion}
For a given $f \in \ell^2(\mathbb{Z})$ satisfying the reversibility symmetry $f_{n} = f_{1-n}$ for every $n \in \mathbb{Z}$,
consider solutions of the linearized difference equation
\begin{equation}
\label{linear-inversion}
-\frac{\alpha_{n+1}-2\alpha_n+\alpha_{n-1}}{\epsilon^2} + \nu^2 (1 - 6 A^2(\epsilon n-\epsilon/2)) \alpha_n = f_n, \quad n \in \mathbb{Z},
\end{equation}
where $A(X) = {\rm sech}(\nu X)$. For sufficiently small $\epsilon > 0$, there exists a unique solution $\alpha \in \ell^2(\mathbb{Z})$
satisfying the reversibility symmetry $\alpha_n = \alpha_{1-n}$ for every $n \in \mathbb{Z}$.
Moreover there is a positive $\epsilon$-independent constant $C$ such that
\begin{equation}
\label{bound-on-inverse}
\epsilon^{-1} \left\| \sigma_+ \alpha - \alpha \right\|_{\ell^2}  \leq C \| f \|_{\ell^2}, \quad \| \alpha \|_{\ell^2} \leq C \| f \|_{\ell^2},
\end{equation}
where $\sigma_+$ is the shift operator defined by $(\sigma_+ \alpha)_n := \alpha_{n+1}$, $n \in \mathbb{Z}$.
\end{proposition}

\begin{proof}
As $\epsilon \to 0$, the finite difference operator in the linearized difference equation (\ref{linear-inversion})
converges to the Schr\"{o}dinger operator
\begin{equation}
\label{linear-operator-2}
L_{\infty} := -\partial_X^2 + \nu^2 (1 - 6 A^2(X)), \quad A^2(X) = {\rm sech}^2(\nu X),
\end{equation}
where $X$ is now defined on the real line. The Schr\"{o}dinger operator (\ref{linear-operator-2})
provides a linearization of the second-order differential equation (\ref{nonlinear-system-6}).
As is well-known (see \cite{text-ode} and references therein),
the spectrum of $L_{\infty}$ consists of the continuous spectrum $\sigma_c(L_{\infty}) \in [\nu^2,\infty)$
and two isolated eigenvalues, one of which is negative
and the other one is at zero. The zero eigenvalue is related to the translational symmetry
and corresponds to the eigenfunction of $L_{\infty}$ spanned by $A'(X)$.
By continuity of isolated eigenvalues with respect to parameter $\epsilon$,
the linearized difference operator in the left-hand side of equation (\ref{linear-inversion}) admits
an eigenvalue near zero, while the rest of its spectrum is bounded away from zero.

Let us now impose the reversibility constraint $\alpha_0 = \alpha_1$ on solutions of the
linearized difference equation (\ref{linear-inversion}). If $f \in \ell^2(\mathbb{Z})$
satisfies the reversibility symmetry $f_{n} = f_{1-n}$ for every $n \in \mathbb{Z}$,
then the constraint $\alpha_0 = \alpha_1$ defines uniquely solutions
of the linearized difference equation (\ref{linear-inversion})
satisfying the reversibility symmetry $\alpha_{n} = \alpha_{1-n}$ for every $n \in \mathbb{Z}$.

Since $A'(-X) = -A'(X)$, for every $X \in \mathbb{R}$,
the Schr\"{o}dinger operator $L_{\infty}$ is invertible on the space of even functions.
Similarly, for sufficiently small values of $\epsilon$, the linearized difference operator is invertible
on sequence $f \in \ell^2(\mathbb{Z})$ satisfying the constraint $f_n = f_{1-n}$ for every $n \in \mathbb{Z}$.
Hence, we obtain the unique reversible
solution $\alpha \in \ell^2(\mathbb{Z})$ to the linearized difference equation (\ref{linear-inversion})
satisfying the second bound in (\ref{bound-on-inverse}).

The first bound is found from the quadratic form associated with the linearized difference equation (\ref{linear-inversion}),
which can be written in the form
\begin{equation}
\label{quad-form}
\epsilon^{-2} \left\| \sigma_+ \alpha - \alpha \right\|_{\ell^2}^2 = \langle \alpha, f \rangle_{\ell^2} + \nu^2
\langle (6 A^2(\epsilon \cdot - \epsilon/2) - 1) \alpha, \alpha \rangle_{\ell^2}.
\end{equation}
By using Cauchy--Schwartz inequality in (\ref{quad-form}) and the second bound in (\ref{bound-on-inverse}), we obtain
the first bound in (\ref{bound-on-inverse}).
\end{proof}

With the help of Proposition \ref{proposition-inversion}, we prove the persistence of the approximation (\ref{slowly-varying}) and
(\ref{approximation-soliton}) among the reversible solutions of the discrete map (\ref{nonlinear-system-4}).

\begin{lemma}
\label{proposition-persistence}
Consider solutions of the discrete map (\ref{nonlinear-system-4}) in the perturbed form
\begin{equation}
\label{perturbed form}
\alpha_n = A(\epsilon n - \epsilon/2 + X_0) + \tilde{\alpha}_n, \quad \beta_n = B(\epsilon n - \epsilon/2 + X_0) + \tilde{\beta}_n, \quad n \in \mathbb{Z},
\end{equation}
where $X_0$ is a parameter. There exists $\epsilon_0 > 0$ and $C_0 > 0$ such that for every $\epsilon \in (0,\epsilon_0)$,
there exist a unique choice for $X_0$ and $(\tilde{\alpha},\tilde{\beta}) \in \ell^2(\mathbb{Z})$ satisfying
\begin{equation}
\label{estimate-final}
|X_0| + \| \tilde{\alpha} \|_{\ell^2} + \| \tilde{\beta} \|_{\ell^2} \leq C_0 \epsilon,
\end{equation}
such that $(\alpha,\beta) \in \ell^2(\mathbb{Z})$ solve the discrete map (\ref{nonlinear-system-4})
subject to the following reversibility constraint on $(\alpha_0,\beta_0)$:
\begin{equation}
\label{symmetry-constraint-1-reduced}
\partial_x \psi(L/2;\epsilon \alpha_0,\epsilon^2 \beta_0,\epsilon) = 0.
\end{equation}
\end{lemma}

\begin{proof}
First, we rewrite the two-dimensional discrete map (\ref{nonlinear-system-4}) as
the scalar second-order difference equation. This equivalent formulation is convenient
for persistence analysis near the approximated solution (\ref{approximation-soliton}). Expressing
\begin{equation}
\label{beta-n-elimination}
\epsilon (L + \pi/2) \beta_n =
\frac{\alpha_{n+1} - \alpha_n - \frac{1}{2} \epsilon^2 (L^2 + L \pi + \pi^2) (1 - 2\alpha_n^2) \alpha_n + \mathcal{O}(\epsilon^4)}{1
+ \epsilon^2 \frac{2 L^3 + 3 L^2 \pi + 6 L \pi^2 + \pi^3}{12 ( L + \pi/2)} (1 - 6 \alpha_n^2) + \mathcal{O}(\epsilon^4)}
\end{equation}
from the first equation of system (\ref{nonlinear-system-4}), we close the second equation
of system (\ref{nonlinear-system-4}) in the form
\begin{equation}
\label{nonlinear-system-7}
-\frac{\alpha_{n+1}-2 \alpha_n + \alpha_{n-1}}{\epsilon^2} + \nu^2 (1 - 2 \alpha_n^2) \alpha_n =
F(\alpha_{n+1},\alpha_n,\alpha_{n-1}) + \epsilon^2 R(\alpha_{n+1},\alpha_n,\alpha_{n-1},\epsilon),
\end{equation}
where
{\small
\begin{eqnarray*}
F(\alpha_{n+1},\alpha_n,\alpha_{n-1}) & := & \frac{1}{2} (L^2 + 4 L \pi + \pi^2)
\left( 1 - 2\alpha_n^2 - 2 \alpha_n \alpha_{n-1} - 2 \alpha_{n-1}^2 \right) (\alpha_n - \alpha_{n-1})  \\
& \phantom{t} & -\frac{2 L^3 + 3 L^2 \pi + 6 L \pi^2 + \pi^3}{12 ( L + \pi/2)} (1 - 6 \alpha_n^2) (\alpha_{n+1} - \alpha_{n}) \\
& \phantom{t} & -\frac{8 L^3 + 24 L^2 \pi + 6 L \pi^2 + \pi^3}{24 ( L + \pi/2)} ( 1 - 6 \alpha_{n-1}^2) (\alpha_n - \alpha_{n-1}).
\end{eqnarray*}
}The remainder term $R$ is a smooth function of $(\alpha_{n+1}\alpha_n,\alpha_{n-1})$ and $\epsilon$, which remains
bounded by an $\epsilon$-independent constant as $\epsilon \to 0$ if $\| \alpha \|_{\ell^2}$ is bounded
by an $\epsilon$-independent constant.

Next, we ensure that the parameter $X_0$ can be uniquely chosen to satisfy the reversibility symmetry (\ref{symmetry-1}).
By Proposition \ref{prop-symmetry}, the symmetry (\ref{symmetry-1}) is satisfied if and only if the parameters $(a_0,b_0)$ satisfies
the constraint (\ref{symmetry-constraint-1}). By virtue of the scaling (\ref{scaling-map}), the reversibility constraint is
written in the form (\ref{symmetry-constraint-1-reduced}).
By Lemma \ref{proposition-expansion}, we rewrite the constraint (\ref{symmetry-constraint-1-reduced})
in the perturbed form
$$
\beta_0 + \frac{\epsilon L}{2} (1 - 2 \alpha_0^2) \alpha_0 + \frac{\epsilon^2 L^2}{8} (1 - 6 \alpha_0^2) \beta_0 + \mathcal{O}(\epsilon^3) = 0,
$$
from which we obtain
\begin{equation}
\label{constraint-beta-2}
\beta_0 = -\frac{\epsilon L}{2} (1 - 2 \alpha_0^2) \alpha_0 + \mathcal{O}(\epsilon^3).
\end{equation}
Using (\ref{beta-n-elimination}), we rewrite this constraint in the form
\begin{equation}
\label{constraint-beta-0}
\alpha_1 - \alpha_0 = \frac{\pi (L + 2\pi) \epsilon^2}{4} (1 - 2 \alpha_0^2) \alpha_0 + \epsilon^4 G(\alpha_0,\epsilon),
\end{equation}
where $G$ is a smooth function of $\alpha_0$ and $\epsilon$, which remains
bounded by an $\epsilon$-independent constant as $\epsilon \to 0$ if
$\alpha_0$ is bounded by an $\epsilon$-independent constant.

Substituting the decomposition (\ref{perturbed form}) into the constraint (\ref{constraint-beta-0}), we obtain
\begin{eqnarray}
\label{constraint-B}
\tilde{\alpha}_1 - \tilde{\alpha}_0 + A(X_0 + \epsilon/2) - A(X_0-\epsilon/2) =
\frac{\pi (L + 2\pi) \epsilon^2}{4} (1 - 2 \alpha_0^2) \alpha_0 + \epsilon^4 G(\alpha_0,\epsilon),
\end{eqnarray}
where $\alpha_0 = A(X_0-\epsilon/2) + \tilde{\alpha}_0$.
For uniqueness of the decomposition (\ref{perturbed form}), we supply the constraint
\begin{equation}
\label{constraint-tilde}
\tilde{\alpha}_1 = \tilde{\alpha}_0.
\end{equation}
Since $A'(0) = 0$ and $A''(0) \neq 0$, we can apply the implicit function theorem to solve
the implicit equation (\ref{constraint-B}) for $X_0$ in terms of $\tilde{\alpha}_0$
for $\epsilon > 0$ sufficiently small. The value of $X_0 = \mathcal{O}(\epsilon)$
is uniquely determined from the implicit equation
for every $\tilde{\alpha} \in \ell^2(\mathbb{Z})$ satisfying a priori bound
\begin{equation}
\label{apriori-assumption}
\| \tilde{\alpha} \|_{\ell^2} \leq C,
\end{equation}
where $C > 0$ is $\epsilon$-independent.

Next, we proceed with the perturbed solution of the discrete map (\ref{nonlinear-system-7}).
By substituting the decomposition (\ref{perturbed form}) and using the differential equation (\ref{nonlinear-system-6}),
we rewrite the second-order difference equation in the equivalent form
\begin{eqnarray}
\nonumber
& \phantom{t} & -\frac{\tilde{\alpha}_{n+1}-2 \tilde{\alpha}_n + \tilde{\alpha}_{n-1}}{\epsilon^2} + \nu^2 (1 - 6 A^2(\epsilon n-\epsilon/2))
\tilde{\alpha}_n \\
\nonumber
& = & H_n(\epsilon) + 6 \nu^2 G_n(\epsilon) \tilde{\alpha}_n
+ 6 \nu^2 A(\epsilon n -\epsilon/2 + X_0) \tilde{\alpha}_n^2 + 2 \nu^2 \tilde{\alpha}_n^3 \\
& \phantom{t} & + F(\alpha_{n+1},\alpha_n,\alpha_{n-1}) + \epsilon^2 R(\alpha_{n+1},\alpha_n,\alpha_{n-1},\epsilon),
\label{nonlinear-system-alpha}
\end{eqnarray}
where
$$
H_n(\epsilon) := \frac{A(\epsilon n + \epsilon/2 + X_0) - 2 A(\epsilon n -\epsilon/2 + X_0) + A(\epsilon n - 3\epsilon/2 + X_0)}{\epsilon^2}
- A''(\epsilon n -\epsilon/2 + X_0).
$$
and
$$
G_n(\epsilon) := A^2(\epsilon n -\epsilon/2 + X_0) - A^2(\epsilon n - \epsilon/2).
$$
Since $A(X) \in C^{\infty}(\mathbb{R})$ and $X_0 = \mathcal{O}(\epsilon)$,
the terms $H(\epsilon)$ and $G(\epsilon)$ satisfy the estimates
\begin{equation}
\label{bounds-1}
\| H(\epsilon) \|_{\ell^2} \leq C \epsilon^2, \quad \| G(\epsilon) \|_{\ell^2} \leq C \epsilon,
\end{equation}
where the positive constant $C$ is $\epsilon$-independent for every $\epsilon > 0$ sufficiently small.

On the other hand, both functions $F(\alpha_{n+1},\alpha_n,\alpha_{n-1})$
and $R(\alpha_{n+1},\alpha_n,\alpha_{n-1},\epsilon)$ in (\ref{nonlinear-system-alpha})
are $C^{\infty}$ in terms of $(\alpha_{n+1},\alpha_n,\alpha_{n-1})$ and $\epsilon$.
It follows from the explicit expression for $F$ that
\begin{equation}
\label{bounds-2}
\| F(A(\epsilon \cdot + \epsilon/2 + X_0),A(\epsilon \cdot - \epsilon/2 + X_0),A(\epsilon \cdot - 3\epsilon/2 + X_0)) \|_{\ell^2} \leq C \epsilon,
\end{equation}
where $C > 0$ is $\epsilon$-independent. Therefore, the inhomogeneous terms
of the perturbed system (\ref{nonlinear-system-alpha}) is bounded in a ball in $\ell^2(\mathbb{Z})$
of the size $\mathcal{O}(\epsilon)$. Since $\epsilon > 0$ is sufficiently small,
vectors in this ball satisfy the a priori assumption (\ref{apriori-assumption}) used earlier.

By using the bounds (\ref{bound-on-inverse}) of Proposition \ref{proposition-inversion},
we look for solutions of the persistence problem (\ref{nonlinear-system-alpha})
satisfying the bounds
\begin{equation}
\label{bound-on-inverse-alpha}
\epsilon^{-1} \left\| \sigma_+ \tilde{\alpha} - \tilde{\alpha} \right\|_{\ell^2}  \leq C \epsilon, \quad
\| \tilde{\alpha} \|_{\ell^2} \leq C \epsilon,
\end{equation}
where $C > 0$ is $\epsilon$-independent. Linearization of $F(\alpha_{n+1},\alpha_n,\alpha_{n-1})$
at $\alpha_n = A(\epsilon n -\epsilon/2 + X_0)$ yields $\mathcal{O}(\epsilon)$ perturbations to the linearized difference operator
in the left-hand side of (\ref{nonlinear-system-alpha}) acting on
$\epsilon^{-1} (\sigma_+ \tilde{\alpha} - \tilde{\alpha})$ and $\tilde{\alpha}$ in $\ell^2(\mathbb{Z})$.

We are now in position to invert the linearized difference operator in the left-hand-side of (\ref{nonlinear-system-alpha})
and to apply the fixed-point iterations for solutions of the persistence problem (\ref{nonlinear-system-alpha})
satisfying the estimates (\ref{bound-on-inverse-alpha}).
Indeed, the solution $\tilde{\alpha}$ is supposed to satisfy the reversibility constraint (\ref{constraint-tilde}),
whereas the right-hand side satisfies the reversibility condition used in Proposition \ref{proposition-inversion},
thanks to our choice of parameter $X_0$ from (\ref{constraint-B}). Equations (\ref{constraint-B}) and (\ref{constraint-tilde})
together define uniquely $X_0$ and satisfy the reversibility constraint (\ref{symmetry-constraint-1-reduced}).

By Proposition \ref{proposition-inversion}, the linearized difference operator in the left-hand-side of (\ref{nonlinear-system-alpha})
is invertible with the bound (\ref{bound-on-inverse}) for sufficiently small $\epsilon > 0$,
so that the fixed-point iterations converge in a ball in $\ell^2(\mathbb{Z})$
of the size $\mathcal{O}(\epsilon)$.
By the implicit function theorem, thanks to the estimates (\ref{bounds-1}) and (\ref{bounds-2}),
we obtain a unique solution $\tilde{\alpha} \in \ell^2(\mathbb{Z})$ to the perturbed system (\ref{nonlinear-system-alpha})
satisfying the reversibility constraint (\ref{constraint-tilde}) and the bounds (\ref{bound-on-inverse-alpha}).
Combining this result with the unique choice for $X_0 = \mathcal{O}(\epsilon)$
from (\ref{constraint-B}) and $\beta \in \ell^2(\mathbb{Z})$ from (\ref{beta-n-elimination}),
we obtain bound (\ref{estimate-final}). The statement of the lemma is proved.
\end{proof}

\begin{remark}
Lemma \ref{proposition-persistence} can be extended to solutions of the discrete map
(\ref{nonlinear-system-4}) satisfying the reversibility symmetry
(\ref{symmetry-2}). In this case, using the scaling (\ref{scaling-map}), we rewrite
the reversibility constraint (\ref{symmetry-constraint-2}) in the form:
\begin{equation}
\label{symmetry-constraint-2-reduced}
\partial_x \psi(\pi/2;\epsilon \gamma_0,\epsilon^2 \delta_0,\epsilon) = 0.
\end{equation}
After the straightforward computations involving system (\ref{gamma-delta-in-terms-of-alpha-beta}),
the constraint (\ref{symmetry-constraint-2-reduced}) can be expressed for the variables $(\alpha_0,\beta_0)$
in the perturbed form:
\begin{equation}
\label{constraint-beta-1}
\beta_0 = -\epsilon (L + \pi) (1 - 2 \alpha_0^2) \alpha_0 + \mathcal{O}(\epsilon^3),
\end{equation}
which is not so different from (\ref{constraint-beta-2}). As a result,
the value of $X_0$ is chosen differently from the constraint
(\ref{constraint-beta-1}), yet, the construction for $X_0$ and $(\tilde{\alpha},\tilde{\beta}) \in \ell^2(\mathbb{Z})$ is unique.
\label{remark-other-symmetry}
\end{remark}

\begin{remark}
It follows from the representation (\ref{correspondence-phi}), Proposition \ref{prop-symmetry},
and Lemma \ref{proposition-expansion} that the results of Lemma \ref{proposition-persistence}
and Remark \ref{remark-other-symmetry} imply the statement of Theorem \ref{theorem-main}
with bound (\ref{branch-bounds}) and symmetries (\ref{branch-1}) and (\ref{branch-2})
on the bound states $\phi$ of the stationary NLS equation (\ref{statNLS}) with $\Lambda = -\epsilon^2$.
Property (i) holds by the construction. However, properties (ii) and (iii) have not yet been proved.
In particular, we cannot state positivity and exponential decay of the sequence $\{ \alpha_n, \beta_n \}_{n \in \mathbb{Z}}$,
which makes it impossible to claim the same properties for the bound state $\phi$.
\end{remark}

\begin{remark}
The result of Lemma \ref{proposition-persistence} corresponds to the reductions
of the NLS equation (\ref{NLS}) on the periodic graph $\Gamma$
to the cubic NLS equation (\ref{nls0}) established in \cite{GilgPS}.
Equation (\ref{nonlinear-system-6}) corresponds to the stationary NLS equation
obtained from the cubic NLS equation (\ref{nls0}) with $\beta = \nu^{-2}$
and $\gamma = 2$ for the steady solutions in the form $\Psi(X,T) = A(X) e^{-i T}$.
Indeed, justification of stationary versions of the NLS
equation follows closely to the justification of the time-dependent equations
and relies on the generalized Lyapunov--Schmidt reduction method \cite{Peli}.
Bloch wave functions are used for derivation and justification of these equations,
whereas the coefficient $\beta = \nu^{-2}$ corresponds to the asymptotic computation (\ref{lowest-band})
obtained from the linear analysis of the lowest spectral band $\sigma_1$
in the spectrum of $-\partial_x^2 : \mathcal{D} \to L^2(\Gamma)$.
\end{remark}

\section{Properties of the bound states bifurcating on $\Gamma$}

We prove here positivity and exponential decay of the sequence $\{ \alpha_n,\beta_n \}_{n \in \mathbb{Z}}$
for the reversible solutions of the two-dimensional discrete map (\ref{nonlinear-system-4}).
To do so, we use the theory of invariant manifolds for discrete maps. Applications
of this theory to construct two distinct sets of so-called on-site and inter-site homoclinic orbits
in the discrete NLS equation can be found in \cite{QinXiao}. A different but spiritually similar
technique for approximations of homoclinic orbits in discrete maps via normal forms is described in
\cite{James}.

We rewrite the two-dimensional discrete map (\ref{nonlinear-system-4}) in the abstract form
\begin{equation}
\label{nonlinear-map}
\left\{ \begin{array}{l} \alpha_{n+1} = \alpha_n + f_{\epsilon}(\alpha_n,\beta_n), \\
\beta_{n+1} = \beta_n + g_{\epsilon}(\alpha_n,\beta_n), \end{array} \right.
\end{equation}
where $(f_{\epsilon},g_{\epsilon})$ are smooth functions of $(\alpha_n,\beta_n)$ and $\epsilon$,
which are available in the form of the perturbative expansion:
\begin{eqnarray}
\label{explicit-vector-field}
\left\{ \begin{array}{l}
f_{\epsilon}(\alpha_n,\beta_n) := \epsilon (L + \pi/2) \beta_n + \frac{1}{2} \epsilon^2 (L^2 + \pi L + \pi^2) (1 - 2 \alpha_n^2) \alpha_n + \mathcal{O}(\epsilon^3), \\
g_{\epsilon}(\alpha_n,\beta_n) := \epsilon (L + 2 \pi) (1 - 2 \alpha_n^2) \alpha_n
+ \frac{1}{4} \epsilon^2 (2 L^2 + 4 L \pi + \pi^2) (1 - 6 \alpha_n^2) \beta_n + \mathcal{O}(\epsilon^3).  \end{array} \right.
\end{eqnarray}
We shall prove existence of homoclinic reversible orbits in the discrete map (\ref{nonlinear-map}) for small $\epsilon \neq 0$,
which have the required properties of positivity and exponential decay.

\begin{lemma}
\label{proposition-properties-map}
There exist $\epsilon_0 > 0$ and $C_0 > 0$ such that for every $\epsilon \in (0,\epsilon_0)$, there exists
two distinct homoclinic orbits to the discrete map (\ref{nonlinear-map}) such that $(\alpha_0,\beta_0)$
satisfy either constraint (\ref{symmetry-constraint-1-reduced}) or (\ref{symmetry-constraint-2-reduced}).
Moreover, for each homoclinic orbit, we have
\begin{equation}
\label{branch-bounds-prop}
\| \alpha \|_{\ell^2} + \| \beta \|_{\ell^2} \leq C_0
\end{equation}
and
\begin{itemize}
\item[(a)] $\alpha_n > 0$ for every $n \in \mathbb{Z}$,
\item[(b)] $\alpha_n \to 0$ as $|n| \to \infty$ exponentially fast,
\item[(c)] there is $N \geq 0$ such that
$\{ \alpha_n \}_{n \in \mathbb{Z}}$ is monotonically increasing for $n \leq -N$ and decreasing for $n \geq N$.
\end{itemize}
The sequence $\{ (\alpha_n,\beta_n) \}_{n \in \mathbb{Z}}$ for the two homoclinic orbits is smooth in $\epsilon$.
\end{lemma}

\begin{proof}
The point $(0,0)$ is a fixed point of the discrete map (\ref{nonlinear-map}) because $f_{\epsilon}(0,0) = g_{\epsilon}(0,0) = 0$ follows by
existence and uniqueness of zero solutions of the initial-value problems (\ref{initial-value-1}) and (\ref{initial-value-2}),
see Proposition \ref{prop-ODE}. The Jacobian matrix of the discrete map (\ref{nonlinear-map}) at $(0,0)$ is given by
\begin{equation}
A = \left[ \begin{array}{cc} 1 + \mathcal{O}(\epsilon^2) & \epsilon (L + \pi/2) + \mathcal{O}(\epsilon^3) \\
\epsilon (L + 2\pi) + \mathcal{O}(\epsilon^3) & 1 + \mathcal{O}(\epsilon^2) \end{array} \right].
\end{equation}
The two eigenvalues of $A$ are $\lambda_{\pm} = 1 \pm \epsilon \nu + \mathcal{O}(\epsilon^2)$,
where $\nu^2 := (L+\pi/2)(L+2\pi)$, therefore, the fixed point is hyperbolic for every $\epsilon \neq 0$.
If there exists a homoclinic orbit $\{ \alpha_n,\beta_n \}_{n \in \mathbb{Z}}$ to the hyperbolic fixed point
$(0,0)$, then $|\alpha_n| + |\beta_n| \to 0$ as $|n| \to \infty$ exponentially fast, by the stable and unstable
curve theory for hyperbolic fixed points.

Therefore, we only need to prove existence of the two distinct positive homoclinic orbits to the discrete map (\ref{nonlinear-map}),
which corresponds to the two reversibility constraints (\ref{symmetry-constraint-1-reduced}) and (\ref{symmetry-constraint-2-reduced}),
see Proposition \ref{prop-symmetry}.
Without loss of generality, we consider the symmetry
constraint (\ref{symmetry-constraint-1-reduced}) given by the asymptotic expansion (\ref{constraint-beta-2}).
The symmetry constraint represents a curve in the $(\alpha,\beta)$ plane given asymptotically by
\begin{equation}
\label{symmetry-curve}
\beta = \mathcal{N}_{\epsilon}(\alpha) := -\frac{\epsilon L}{2} (1 - 2 \alpha^2) \alpha + \mathcal{O}(\epsilon^3) \quad \mbox{\rm as} \quad \epsilon \to 0.
\end{equation}
By the symmetry, it is sufficient to construct the one-dimensional unstable curve for the discrete
map (\ref{nonlinear-map}) and to prove that this curve intersects the curve (\ref{symmetry-curve}).
By discrete group of translations and continuous dependence on the initial conditions,
the intersection point can be translated to the location $n = 0$ in the sequence
$\{ \alpha_n,\beta_n \}_{n \in \mathbb{Z}}$ and hence satisfy the symmetry constraint (\ref{symmetry-constraint-1-reduced}).

For the unstable eigenvalue $\lambda_+ = 1 + \epsilon \nu + \mathcal{O}(\epsilon^2)$, the corresponding
eigenvector of $A$ defines a straight line on the $(\alpha,\beta)$ plane,
\begin{equation}
\label{eigenvector-A}
\beta = \mathcal{U}_{\epsilon}(\alpha) := \left[ \mu + \mathcal{O}(\epsilon) \right] \alpha, \quad \mbox{\rm as} \quad \epsilon \to 0,
\end{equation}
where $\mu^2 := \frac{L+2\pi}{L+\pi/2}$. The straight line (\ref{eigenvector-A})
is located in the first quadrant of positive $\alpha$ and $\beta$. By the invariant curve theory, there exists a one-dimensional unstable
curve in the $(\alpha,\beta)$ plane, which is tangent to the line defined by (\ref{eigenvector-A}).
Therefore, there exists $N_1 \in \mathbb{Z}$ such that $\{ \alpha_n,\beta_n\}_{n = -\infty}^{n=N_1}$ are monotonically increasing
and $\alpha_n, \beta_n > 0$ for every $n \leq N_1$.

Figure \ref{fig:phasePlane} shows the plane $(\alpha,\beta)$ together with the symmetry curve (\ref{symmetry-curve}) (red dash-dotted line) and
the straight line (\ref{eigenvector-A}) (green dashed line). The blue dots show the monotonically increasing sequence
$\{ \alpha_n,\beta_n\}_{n = -\infty}^{n=N_1}$ according to the approximate solution (\ref{approximation-soliton}).

By the explicit form of $f_{\epsilon}$ in (\ref{explicit-vector-field}),
there is a positive $\epsilon$-independent constant $C_1$ such that
$f_{\epsilon}(\alpha_n,\beta_n) > 0$ for $\epsilon > 0$ sufficiently small,
as long as $\beta \geq C_1 \epsilon \alpha$. Therefore, the sequence $\{ \alpha_n \}$
remains monotonically increasing as long as $\beta \geq C_1 \epsilon \alpha$.

By the explicit form of $g_{\epsilon}$ in (\ref{explicit-vector-field}), there is a positive
$\epsilon$-independent constant $C_2$ such that
the sequence $\{ \beta_n \}$ is monotonically decreasing if $\alpha > \frac{1}{\sqrt{2}} - C_2 \epsilon \beta$.

Since
\begin{equation}
\label{extension}
\lambda_+^n = (1 + \epsilon \nu + \mathcal{O}(\epsilon^2))^n
= e^{n \epsilon \nu (1 + \mathcal{O}(\epsilon))}, \quad \mbox{\rm as} \quad \epsilon \to 0,
\end{equation}
it follows from (\ref{eigenvector-A}) and (\ref{extension}) that
there exists $N_1 = \mathcal{O}(\epsilon^{-1})$ such that
$\alpha_{N_1} > \frac{1}{\sqrt{2}} - C_2 \epsilon \beta_{N_1}$ and $\beta_{N_1} = \mathcal{O}(1)$ as $\epsilon \to 0$.
From the explicit forms of $f_{\epsilon}$ and $g_{\epsilon}$ discussed above,
it follows that there exists $N_2 > N_1$ such that
$\{ \alpha_n \}_{n = -\infty}^{n = N_2}$ is monotonically increasing and
$\{ \beta_n \}_{n = N_1}^{n = N_2}$ is monotonically decreasing until it reaches $\beta_{N_2} \leq C_1 \epsilon \alpha_{N_2}$.

Since the values of $\beta_n$ decrease from $\mathcal{O}(1)$ at $n = N_1$ to $\mathcal{O}(\epsilon)$ at $n = N_2$,
we have $N_2 - N_1 = \mathcal{O}(\epsilon^{-1})$ and $\alpha_{N_2} \geq C_3$ for some $C_3 > \frac{1}{\sqrt{2}}$.
If $\beta_{N_2} \leq \mathcal{N}_{\epsilon}(\alpha_{N_2})$, we are done, as the
existence of a homoclinic orbit follows from the discrete translational invariance,
the reversibility symmetry, and the continuous dependence of sequences from initial data.

If $\beta_{N_2} > \mathcal{N}_{\epsilon}(\alpha_{N_2})$,
then we continue iterations further. Since $\alpha_{N_2} \geq C_3 > \frac{1}{\sqrt{2}} - C_2 \epsilon \beta_{N_2}$
and $\beta_{N_2} \leq C_1 \epsilon \alpha_{N_2}$, there is an
$\epsilon$-independent positive integer $K$ such that $\{ \alpha_n \}_{n = N_2}^{n = N_2 + K}$ becomes monotonically
decreasing, whereas $\{ \beta_n \}_{n = N_2}^{n=N_2+K}$ continues to be monotonically decreasing
until it reaches $\beta_{N_2+K} \leq \mathcal{N}_{\epsilon}(\alpha_{N_2+K})$ for
$\alpha_{N_2+K} \geq C_4$ for some $C_4 \in \left(\frac{1}{\sqrt{2}},C_3\right)$.
Again, the existence of a homoclinic orbit follows from the discrete translational invariance,
the reversibility symmetry, and the continuous dependence of sequences from initial data.

A similar construction holds for the symmetry constraint (\ref{symmetry-constraint-2-reduced}),
which is given by the asymptotic expansion (\ref{constraint-beta-1}).
Smoothness in $\epsilon$ is obtained by smoothness of the vector field and all asymptotic expansions in $\epsilon$.
\end{proof}

\begin{figure}[htbp]
   \centering
   \includegraphics[width=4in]{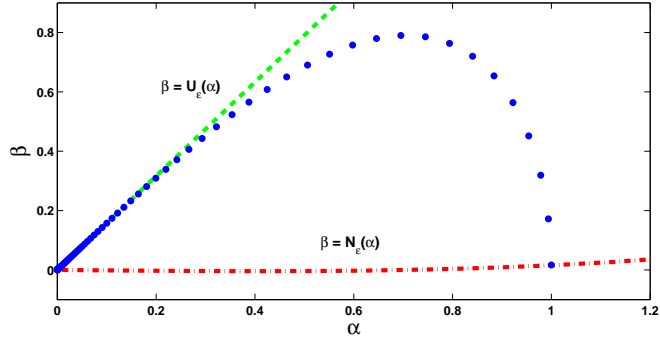}
\caption{The plane $(\alpha,\beta)$ for $L = \pi/2$ and $\epsilon = 0.02$, where
the blue dots denote a sequence $\{ \alpha_n,\beta_n \}_{n \in \mathbb{Z}}$ given approximately by
 the solution (\ref{approximation-soliton}), the green dashed line shows the straight line $\beta = \mathcal{U}_{\epsilon}(\alpha)$,
 and the red dash-dotted line shows the symmetry curve $\beta = \mathcal{N}_{\epsilon}(\alpha)$.}
\label{fig:phasePlane}
\end{figure}

\begin{remark}
If $K > 0$ in the proof of Lemma \ref{proposition-properties-map},
the sequence $\{ \alpha_n \}_{n \in \mathbb{Z}}$ constructed for two distinct homoclinic orbits in
Lemma \ref{proposition-properties-map} is not monotonically increasing for $n < 0$ and
decreasing for $n > 0$. Indeed, at the symmetry curve (\ref{symmetry-curve}), we have
$$
f_{\epsilon}(\alpha,\mathcal{N}_{\epsilon}(\alpha)) = \frac{\pi (L + 2 \pi) \epsilon^2}{4} (1 - 2 \alpha^2) \alpha
+ \mathcal{O}(\epsilon^4),
$$
therefore, $f_{\epsilon}(\alpha,\mathcal{N}_{\epsilon}(\alpha)) < 0$ for $\alpha > \frac{1}{\sqrt{2}}$,
if $\epsilon$ is sufficiently small. Thus, the sequence $\{ \alpha_n \}_{n \in \mathbb{Z}}$ is only proved
to be monotonically increasing for $n < -N$ and decreasing for $n > N$, where the number $N \geq 0$ is $\epsilon$-independent
for sufficiently small $\epsilon$.
\end{remark}

\begin{remark}
Since we are not able to prove that $\alpha_0 = 1 + \mathcal{O}(\epsilon)$ for the two distinct homoclinic
orbits constructed in Lemma \ref{proposition-properties-map}, we are not able to claim immediately that these homoclinic
orbits are the same as the ones in Lemma \ref{proposition-persistence}. Nevertheless, because of convergence
of finite differences to derivatives in the two-dimensional discrete map (\ref{nonlinear-map}) as $\epsilon \to 0$
and smoothness of the sequences $\{ (\alpha_n,\beta_n) \}_{n \in \mathbb{Z}}$ with respect to $\epsilon$ in
Lemma \ref{proposition-properties-map},
the two distinct homoclinic orbits converge as $\epsilon \to 0$ to the leading-order approximation (\ref{approximation-soliton}),
which is uniquely continued in Lemma \ref{proposition-persistence}. Therefore, in the end,
we confirm that these two distinct solutions of Lemmas \ref{proposition-persistence} and \ref{proposition-properties-map} are
the same.
\end{remark}

We shall now transfer properties of Lemma \ref{proposition-properties-map} to the bound states
of the stationary NLS equation (\ref{statNLS-eps}) for small $\epsilon \neq 0$.

\begin{corollary}
\label{corollary-main}
There exist $\epsilon_0 > 0$ and $C_0 > 0$ such that for every $\epsilon \in (0,\epsilon_0)$, there exists
two distinct bound states $\phi \in \mathcal{D} \subset L^2(\Gamma)$
to the stationary NLS equation (\ref{statNLS-eps}) satisfying
the constraints (\ref{branch-1}) and (\ref{branch-2}). Moreover, for each branch, we have
\begin{equation}
\label{branch-bounds-cor}
\| \phi \|_{H^2(\Gamma)} \leq C_0 \epsilon
\end{equation}
and
\begin{itemize}
\item[(i)] $\phi_{n,+}(x) = \phi_{n,-}(x)$ for every $x \in [nP+L,(n+1)P]$,
\item[(ii)] $\phi(x) > 0$ for every $x \in \Gamma$,
\item[(iii)] $\phi(x) \to 0$ as $|x| \to \infty$ exponentially fast.
\end{itemize}
The function $\phi$ for the two bound states is smooth in $\epsilon$.
\end{corollary}

\begin{proof}
The result follows from Lemma \ref{proposition-properties-map}, with the representation (\ref{correspondence-phi})
and the asymptotic expansions of Lemma \ref{proposition-expansion}. These asymptotic expansions can be rewritten
explicitly as follows:
\begin{equation*}
\phi_{n,0}(x) = \epsilon \left[ \alpha_n +
\epsilon \beta_n (x-nP) + \frac{1}{2} \epsilon^2 \alpha_n (1 - 2 \alpha_n^2) (x-nP)^2
+ \mathcal{O}_{L^{\infty}(nP,nP+L)}(\epsilon^3) \right]
\end{equation*}
and
\begin{equation*}
\phi_{n,+}(x) = \epsilon \left[ \gamma_n +
\epsilon \delta_n (x-nP-L) + \frac{1}{2} \epsilon^2 \gamma_n (1 - 2 \gamma_n^2) (x-nP-L)^2
+ \mathcal{O}_{L^{\infty}(nP+L,(n+1)P)}(\epsilon^3) \right],
\end{equation*}
where $\{ (\alpha_n,\beta_n) \}_{n \in \mathbb{Z}}$ is defined at
one of the two homoclinic orbits of Lemma \ref{proposition-properties-map},
whereas $\{ (\gamma_n,\delta_n) \}_{n \in \mathbb{Z}}$ are found from system (\ref{gamma-delta-in-terms-of-alpha-beta}).

Bound (\ref{branch-bounds-cor}) follows from this asymptotic representation and bound (\ref{branch-bounds-prop}).
Symmetry (i) follows by the construction of the reduction (\ref{reduction-2}). Positivity (ii) and exponential
decay (iii) follows from properties (a) and (b) of Lemma \ref{proposition-properties-map} provided
the rate of exponential decay of sequences $\{ \alpha_n \}_{n \in \mathbb{Z}}$ and $\{ \beta_n \}_{n \in \mathbb{Z}}$ coincides.
The latter fact follows from the invariant curve theory for two-dimensional discrete maps.
Smoothness of $\phi$ in $\epsilon$ follows from smoothness of the sequence
$\{ (\alpha_n,\beta_n) \}_{n \in \mathbb{Z}}$ in $\epsilon$, as stated in Lemma \ref{proposition-properties-map}.
\end{proof}

\begin{remark}
Theorem \ref{theorem-main} is a reformulation of the result of Corollary \ref{corollary-main}.
\end{remark}

\section{Discussions}

In order to illustrate the construction of two distinct bound states in Theorem \ref{theorem-main},
we compute numerical approximations of bound states of the stationary NLS equation
(\ref{statNLS-eps}) on the periodic graph $\Gamma$. The numerical method uses the first-order invariant
(\ref{first-order-invariant}) piecewise on the segments $\Gamma_{n,0}$ and $\Gamma_{n,+}$
of the periodic graph $\Gamma$. The starting point $(\phi,\phi') = (\phi_0,0)$
at the center of symmetry (either $x_0 = L/2$ or $x_0 = L + \pi/2$) is used as parameter
of the shooting method. The boundary conditions (\ref{keq1-red}) and (\ref{keq4-red})
are preserved at each breaking point of the graph in order to recompute
the value of the first-order invariant (\ref{first-order-invariant}) and
to continue the solution away from the points of symmetry. The parameter $\phi_0$
of the shooting method is adjusted to obtain a homoclinic orbit on the periodic graph $\Gamma$.
The two distinct families of the bound states are shown on Figures \ref{fig:state1} and \ref{fig:state2}
both on the plane $(x,\phi)$ and on the phase plane $(\phi,\phi')$.
The green dotted line shows the solitary wave of the stationary NLS equation
on the infinite line, which is available analytically as $\phi(x) = \epsilon {\rm sech}(\epsilon (x-x_0))$
for each of the bound state. Note that although the jumps in the derivative $\phi'$ look
large on the right panels of Figures \ref{fig:state1} and \ref{fig:state2},
the scaling is $\mathcal{O}(10^{-3})$ for the vertical axis.

\begin{figure}[htbp]
   \centering
   \includegraphics[width=3.in]{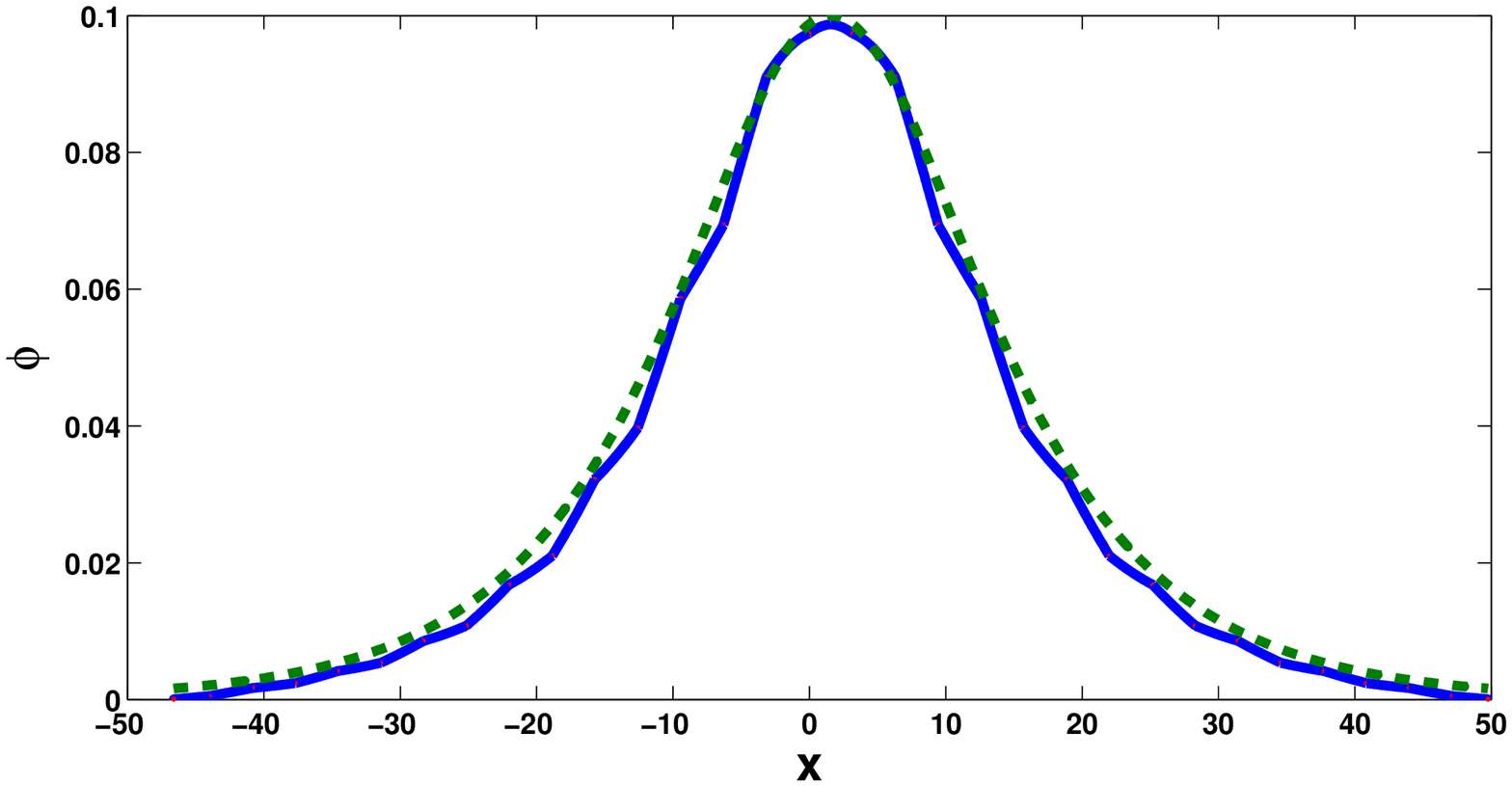}
   \includegraphics[width=3.in]{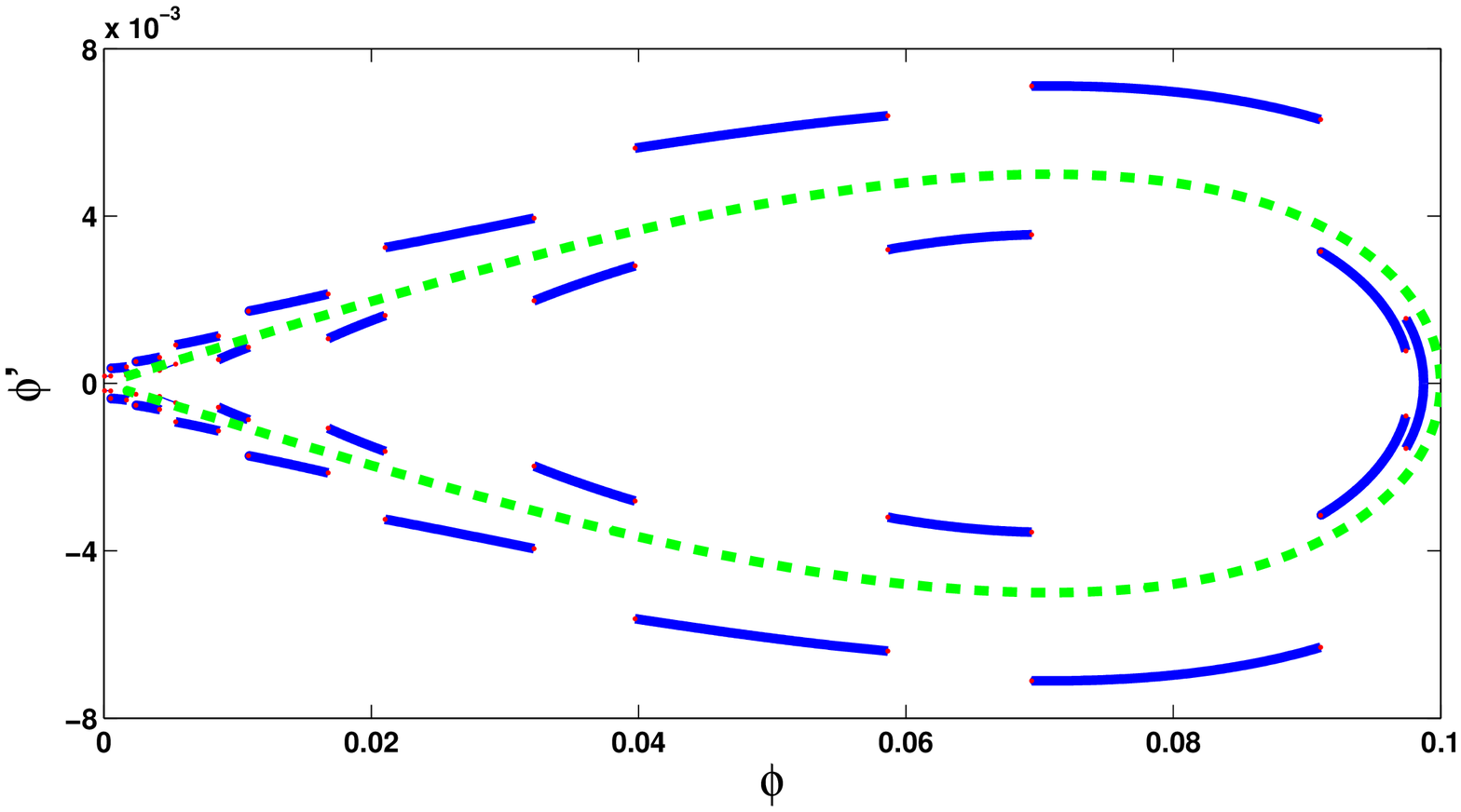}
\caption{Profile of the numerically generated bound state of the stationary NLS equation (\ref{statNLS-eps})
with symmetry (\ref{branch-1}) for $L = \pi$ and $\epsilon = 0.1$ on $(x,\phi)$ plane (left) and on $(\phi,\phi')$ plane
(right). The red dots show the breaking points on the periodic graph $\Gamma$. The green dashed line shows
the solitary wave solution of the stationary NLS equation (\ref{statNLS-eps}) on the infinite line.}
\label{fig:state1}
\end{figure}

\begin{figure}[htbp]
   \centering
   \includegraphics[width=3.in]{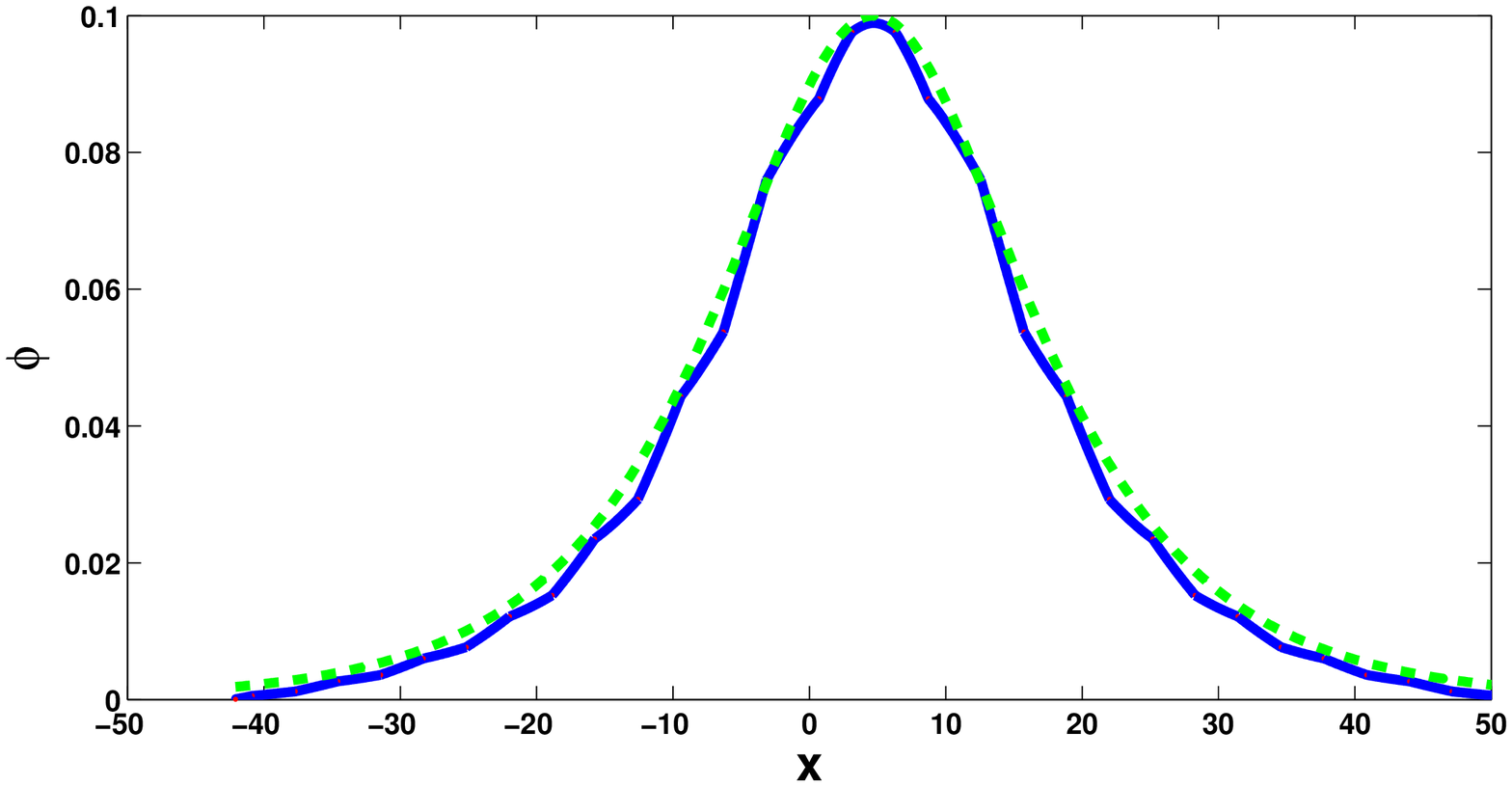}
   \includegraphics[width=3.in]{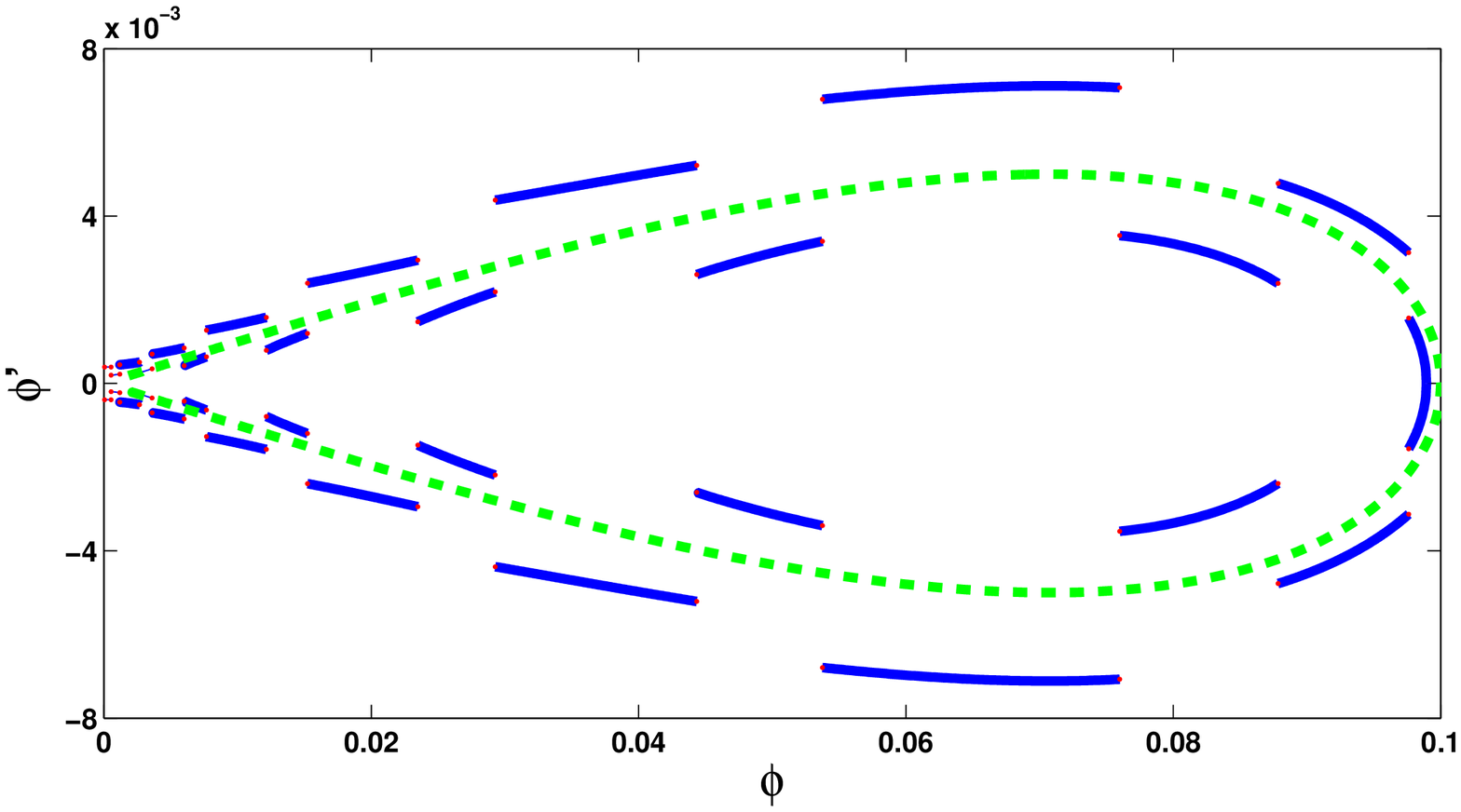}
\caption{The same as Figure \ref{fig:state1} but for the bound state with symmetry (\ref{branch-2}).}
\label{fig:state2}
\end{figure}

We end this work with remarks on energy estimates, which are useless to distinguish which of the two bound
states of Theorem \ref{theorem-main} serves a ground state. By definition, the ground state is a
bound state of minimal energy $E$ given by (\ref{energy}) under fixed mass $Q$ given by (\ref{charge}). Equivalently,
the ground state is a solution $\phi \in \mathcal{D} \subset L^2(\Gamma)$ to
the stationary NLS equation (\ref{statNLS-eps}) for fixed $\epsilon > 0$ that corresponds to the smallest value of mass $Q$.

Substituting the asymptotic representation of Corollary \ref{corollary-main} to the expression (\ref{charge})
and using smoothness of all quantities in $\epsilon$, we obtain the asymptotic representation for each
homoclinic orbit:
\begin{equation}
\label{Q-expansion}
Q = \epsilon^2 \sum_{n \in \mathbb{Z}} (L \alpha_n^2 + 2 \pi \gamma_n^2) +
\epsilon^3 \sum_{n \in \mathbb{Z}} (L^2 \alpha_n \beta_n + 2 \pi^2 \gamma_n \delta_n) + \mathcal{O}(\epsilon^4).
\end{equation}
Symmetry (\ref{symmetry-1}) of the first homoclinic orbit implies
$$
\alpha_{-n} = \gamma_n, \quad \beta_{-n} = -2 \delta_n, \quad n \in \mathbb{Z},
$$
which reduces the asymptotic expression (\ref{Q-expansion}) for $Q$ to the form
\begin{equation}
\label{expression-Q}
Q = (L+2\pi) \epsilon^2 \sum_{n \in \mathbb{Z}} \alpha_n^2 + (L^2-\pi^2) \epsilon^3 \sum_{n \in \mathbb{Z}} \alpha_n \beta_n + \mathcal{O}(\epsilon^4).
\end{equation}
On the other hand, symmetry (\ref{symmetry-2}) of the second homoclinic orbit implies
$$
\alpha_{-n} = \gamma_{n+1}, \quad \beta_{-n} = -2 \delta_{n+1}, \quad n \in \mathbb{Z},
$$
which results in the same expression (\ref{expression-Q}) for $Q$. Since $\{ (\alpha_n,\beta_n) \}_{n \in \mathbb{Z}}$ approaches the
same approximation (\ref{approximation-soliton}) as $\epsilon \to 0$, the value of $Q$ is identical
in the first two orders of $\epsilon$. In fact, we suspect that this value is only different by a term
that is exponentially small in $\epsilon$. As a result, we are not able to claim which bound state is
a ground state by using power expansions in $\epsilon$.

\begin{figure}[htbp]
   \centering
   \includegraphics[width=3.in]{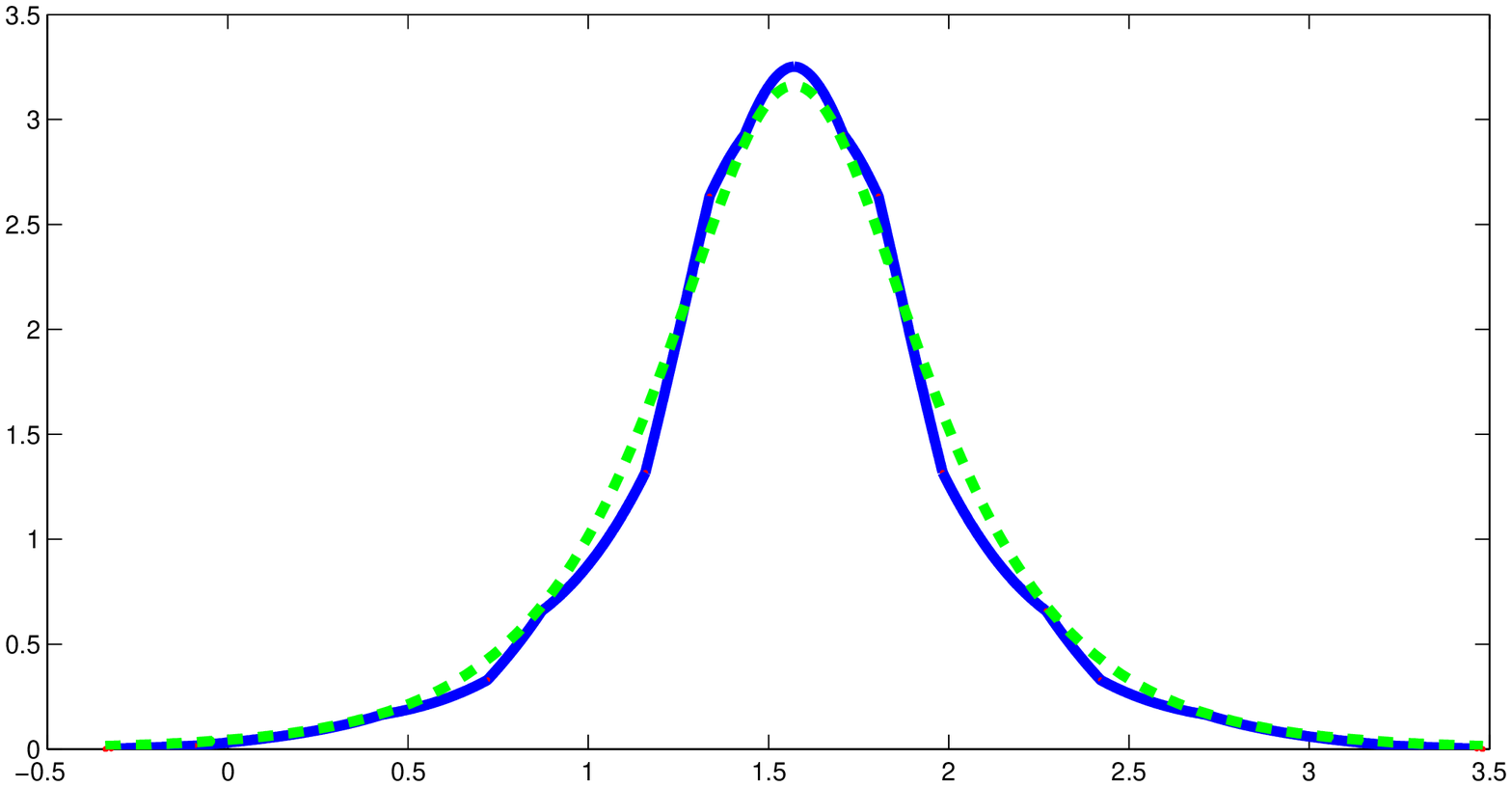}
   \includegraphics[width=3.in]{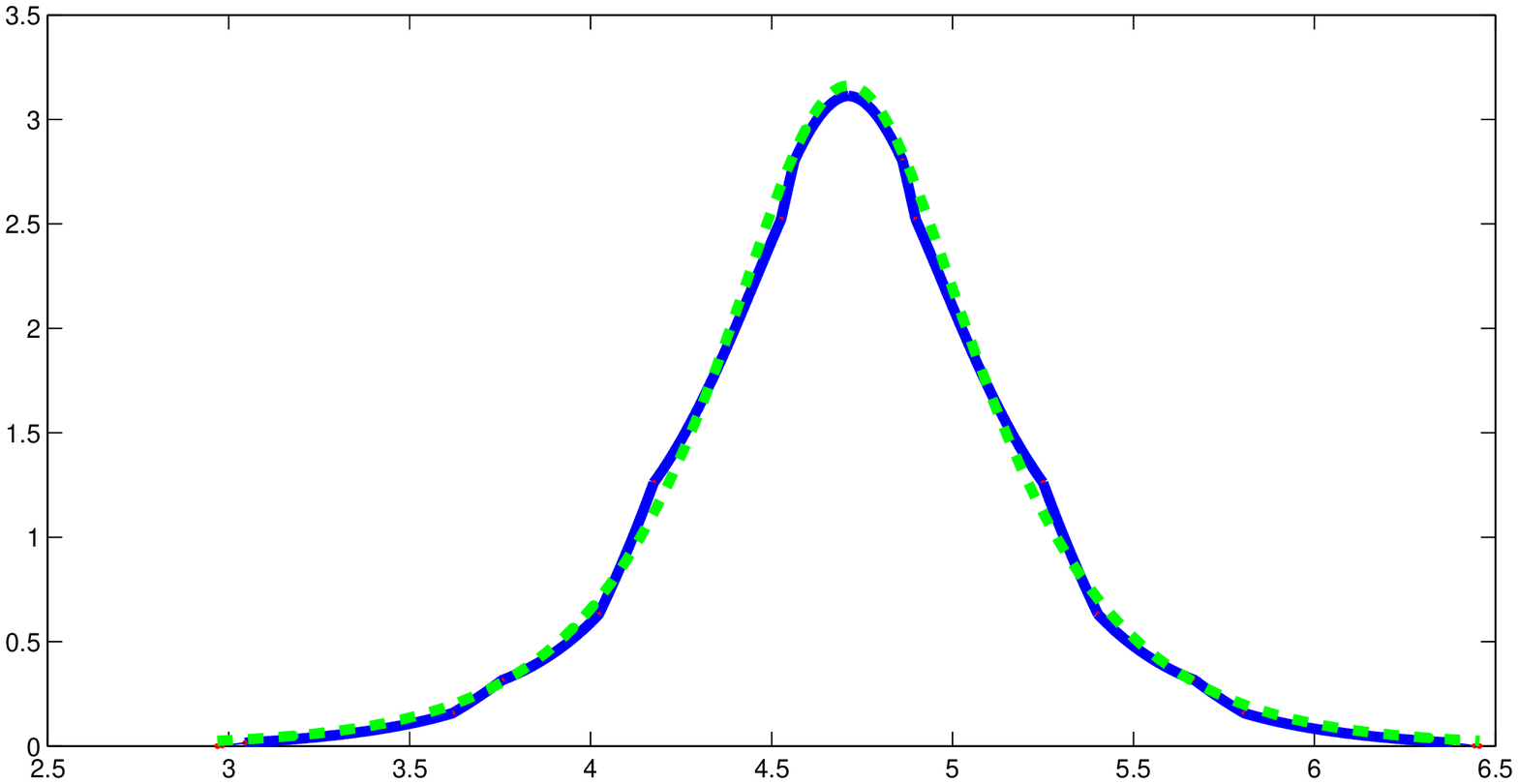}
\caption{Profile of the numerically generated bound state of the stationary NLS equation (\ref{statNLS})
with symmetries (\ref{branch-1}) (left) and (\ref{branch-2}) (right) for $L = \pi$ and $\Lambda = -10$.
The green dashed line shows the solitary wave solution of the stationary NLS equation (\ref{statNLS}) on the infinite line.}
\label{fig:state3}
\end{figure}

As is formulated in Remarks \ref{remark-1} and \ref{remark-2},
we anticipate that the bound state satisfying the symmetry (\ref{branch-1}) is a ground state of the
periodic metric graph. We also anticipate that both bound states are extended to all values of $\epsilon$,
that is, to the limit $\Lambda \to -\infty$ within the stationary NLS equation (\ref{statNLS}).
Indeed, using the same numerical method, we have confirmed that both bound states exists
in the stationary NLS equation (\ref{statNLS}) with $\Lambda = -10$.
Figure \ref{fig:state3} illustrates the profiles of the two bound states for $\Lambda = -10$.
In this case, the bound states become more concentrated at the nearest cells
to the symmetry centers of the periodic graph $\Gamma$.

\vspace{0.25cm}

{\bf Acknowledgement.} D. Pelinovsky is grateful to the Humboldt Foundation for sponsoring
his stay at the University of Stuttgart during June-July 2015.
The work of G. Schneider is supported by the Deutsche Forschungsgemeinschaft DFG through the
Research Training Center GRK 1838 ``Spectral Theory and Dynamics of Quantum Systems".

\end{document}